\documentclass[10pt,english]{smfart}

\usepackage{amsmath,xy,leftidx,amsthm,sfheaders}
\usepackage{xspace}
\usepackage[psamsfonts]{amssymb}
\usepackage[latin1]{inputenc}
\usepackage{graphicx,color,mathrsfs}
\usepackage{hyperref}

\theoremstyle{plain}

\newtheorem{theorem}{\bf Theorem}
\newtheorem{proposition}[theorem]{\bf Proposition}
\newtheorem{definition}[theorem]{\bf Definition}

\newtheorem{lemma}[theorem]{\bf Lemma}
\newtheorem{corollary}[theorem]{\bf Corollary}

\def\C{{\mathbb C}}

\def\R{{\mathbb R}}

\def\bif{\textup{bif}}

\title[Good height functions on quasi-projective varieties]{Good height functions on quasi-projective varieties: equidistribution and applications in dynamics}
\author{Thomas Gauthier}
\address{Laboratoire de Math\'ematiques d'Orsay, Universit\'e Paris-Saclay 
91405 Orsay Cedex, France}
\email{thomas.gauthier1@universite-paris-saclay.fr}
\thanks{The author is partially supported by the ANR grant Fatou ANR-17-CE40-0002-01.}
\begin{document}

\begin{abstract}
In the present article, we define a notion of \emph{good height functions} on quasi-projective varieties $V$ defined over number fields and prove an equidistribution theorem of small points for such height functions. Those good height functions are defined as limits of height functions associated with semi-positive adelic metrization on big and nef $\mathbb{Q}$-line bundles on projective models of $V$ satisfying mild assumptions.

Building on a recent work of the author and Vigny as well as on a classical estimate of Call and Silverman, and inspiring from recent works of K\"uhne and Yuan and Zhang, we deduce the equidistribution of generic sequence of preperiodic parameters for families of polarized endomorphisms with marked points.
\end{abstract}

\maketitle


\section*{Introduction}

Let $X$ be a projective variety defined over number field $\mathbb{K}$ and $L$ be an ample line bundle on $X$.
When $L$ is endowed with an adelic semi-positive continuous metric $\{\|\cdot\|_v\}_{v\in M_\mathbb{K}}$ with induced height function $h_{\bar{L}}:X(\bar{\mathbb{Q}})\to\mathbb{R}$, a fundamental result is the existence of a systematic equidistribution of small and generic sequences: if $\{x_n\}_n$ is a sequence of points of $X(\bar{\mathbb{Q}})$ such that $h_{\bar{L}}(x_n)\to h_{\bar{L}}(X)$ and if for any subvariety $Z\subset X$ defined over $\mathbb{K}$ there is $n_0\geq1$ such that the Galois orbit $\mathsf{O}(x_n)$ of $x_n$ is disjoint from $Z$ for all $n\geq n_0$, Yuan~\cite{yuan} proved that for any place $v\in M_\mathbb{K}$, we have 
\[\frac{1}{\deg(x_n)}\sum_{x\in\mathsf{O}(x_n)}\delta_x\rightarrow \frac{c_1(\bar{L})^{\dim X}_v}{\mathrm{vol}(L)},\]
in the weak sense of probability measures on the Berkovich analytic space $X_v^\mathrm{an}$.

\medskip

This result $-$ as well as previous existing results concerning the equidistribution of small points $-$ has shown many important implications in arithmetic geometry and dynamics. Historically, a first striking example is the proof by Ullmo~\cite{Ullmo-Bogomolov} and Zhang~\cite{Zhang-Bogomolov} of the Bogomolov conjecture. An emblematic example in dynamics is the following: let $f_t(z)=z^d+t$ for $(z,t)\in\mathbb{C}^2$ and pick any two complex numbers $a,b\in \mathbb{C}$. Baker and DeMarco~\cite{BD-unlikely} prove that the set of parameters $t\in\C$ such that $a$ and $b$ are both preperiodic points of $f_t$ is infinite if and only if $a^d=b^d$. Building on this work, they propose in~\cite{BD} a dynamical analogue of the Andr\'e-Oort conjecture. 
Let us mention that, relying also on Yuan's Theorem, Favre and the author~\cite{book-unlikely} recently proved this so-called Dynamical Andr\'e-Oort conjecture for curves of polynomials.

\medskip

When trying to prove this conjecture for general families of rational maps $f_t:\mathbb{P}^1\to\mathbb{P}^1$, this strategy fails for several reasons. Given such a family parametrized by a quasi-projective curve together with a marked point $a:S\to\mathbb{P}^1$ (viewed as a moving dynamical point), we still have a candidate height function. However, we don't even know whether this function is a Weil height associated with an $\mathbb{R}$-divisor. Worse, in some cases when we can build a metrized line bundle inducing this height function, the continuity of the metric fails~\cite{DeMarco-Okuyama} or the metric is not anymore adelic~\cite{DeMarco-Wang-Ye}. 

\medskip

In the present article, we introduce a notion of good height function on a quasi-projective variety defined over a number field and prove an equidistribution of small points for such heights, allowing us for example to prove a general equidistribution statement in families of polarized endomorphisms of projective varieties with marked points, which applies in particular in the above mentionned cases where Yuan's result does not apply.

\subsection*{Good height functions and equidistribution}

Let $V$ be a smooth quasi-projective variety defined over a number field $\mathbb{K}$ and place $v\in M_\mathbb{K}$ and let $h:V(\bar{\mathbb{Q}})\to\mathbb{R}$ be a function. A sequence $(F_i)_i$ of Galois-invariant finite subsets of $V(\bar{\mathbb{Q}})$ is $h$-\emph{small} if
\begin{center}
$h(F_i):=\frac{1}{\# F_i}\sum_{x\in F_i}h(x)\to0$, as $i\to\infty$.
\end{center}

\begin{definition}\normalfont
We say $h$ is a \emph{good height} at $v$ if for any $n\geq0$, there is a projective model $X_n$ of $V$ together with a birational morphism $\psi_n:X_n\to X_0$ which is an isomorphism above $V$ and  a big and nef $\mathbb{Q}$-line bundle $L_n$ on $X_n$ endowed with an adelic semi-positive continuous metrization $\bar{L}_n$,
such that the following holds :
\begin{enumerate}
\item[$(1)$] For any generic $h$-small sequence $(F_i)_i$ of Galois-invariant finite subsets of $V(\bar{\mathbb{Q}})$, the sequence
$\varepsilon_n(\{F_i\}_i):=\limsup_ih_{\bar{L}_n}(\psi_n^{-1}(F_i))-h_{\bar{L}_n}(X_n)$ satisfies $\varepsilon_n(\{F_i\})\to0$ as $n\to\infty$,
\item[$(2)$] the sequence of volumes $\mathrm{vol}(L_n)$ converges to $\mathrm{vol}(h)>0$ as $n\to\infty$ and if $c_1(\bar{L}_n)_v$ is the curvature form of $\bar{L}_n$ on $X_{n,v}^{\mathrm{an}}$, then the sequence of finite measures $\left(\mathrm{vol}(L_n)^{-1}(\psi_n)_*c_1(\bar{L}_n)_v^k\right)_n$ converges weakly on $V_v^{\mathrm{an}}$ to a probability measure $\mu_v$,
\item[$(3)$] If $k:=\dim V>1$, for any ample line bundle $M_0$ on $X_0$ and any adelic semi-positive continuous metrization $\bar{M}_0$ on $M_0$, there is a constant $C\geq0$ such that 
\[\left(\psi_n^*(\bar{M}_0)\right)^j\cdot \left(\bar{L}_n\right)^{k+1-j}\leq C,\]
for any $2\leq j\leq k+1$ and any $n\geq0$.
\end{enumerate}
We say that $\mathrm{vol}(h)$ is the \emph{volume} of $h$ and that $\mu_v$ is the measure \emph{induced by $h$} at the place $v$.
We finally say $h$ is a \emph{good height} if it is $v$-good for all $v\in M_\mathbb{K}$. In this case, we say $\{\mu_v\}_{v\in M_\mathbb{K}}$ is the \emph{global measure} induced by $h$.
\end{definition}

 We prove here the following general equidistribution result.
 
\begin{theorem}[Equidistribution of small points]\label{tm:equidistrib}
Let $V$ be a smooth quasiprojective variety defined over a number field $\mathbb{K}$, let $v\in M_\mathbb{K}$ and let $h$ be a $v$-good height on $V$ with induced measure $\mu_v$. For any $h$-small sequence $(F_m)_m$ of Galois-invariant finite subsets of $V(\bar{\mathbb{Q}})$ such that for any hypersurface $H\subset V$ defined over $\mathbb{K}$, we have
\[\# (F_n\cap H)=o(\# F_n), \quad \text{as} \ n\to+\infty,\]
the probability measure $\mu_{F_m,v}$ on $V_v^\mathrm{an}$ which is  equidistributed on $F_m$ converges to $\mu_v$ in the weak sense of measures, i.e. for any continuous function with compact support $\varphi\in\mathscr{C}^0_c(V_v^\mathrm{an})$, we have
\[\lim_{m\to\infty}\frac{1}{\# F_m}\sum_{y\in F_m}\varphi(y)=\int_{V_v^\mathrm{an}}\varphi\,\mu_v.\]
\end{theorem}

This result is inspired by K\"uhne's work~\cite{Kuhne} where he establishes an equidistribution statement in families of abelian varieties and from \cite{YZ-adelic}, where Yuan and Zhang develop a general theory of adelic line bundles on quasi-projective varieties. Among other results, Yuan and Zhang prove an equidistribution theorem for small point in this context. They also deduce Theorem~\ref{tm:principal} and Theorem~\ref{cor:ratd} from this general result. The aim of this article is to provide a more naive and independent approach which seems particularly adapted to a dynamical setting. I also have to mention that, even though he focuses on the case of families of abelian varieties in his paper, K\"uhne has a strategy to generalize his relative equidistribution theorem to a dynamical setting.

\smallskip

It is worth mentioning there are many arithmetic equidistribution statement for small points in the past decades see, e.g., \cite{Ullmo-Zhang,Bilu,rumely,Thuillier,Baker-Rumely,ACL,FRL,yuan,CL-T,BR-book,Berman-Boucksom,mavraki-ye-quasiadelic,BGPRLS}. It si also worth mentioning that, prior to K\"uhne's recent work only the equidistribution results of arithmetic nature from~\cite{rumely} and~\cite{BR-book} do not rely on the continuity of the underlying metrics and compactness of the variety. Also, in both~\cite{rumely} and~\cite{BR-book}, the results are stated on $\mathbb{P}^1$ and each metric is continuous outside a polar set and is bounded. The generalization of their approach remains unexplored in a more general context. It should also be noted that Mavraki and Ye~\cite{mavraki-ye-quasiadelic} were the first to get rid of the assumption that the metrization is \emph{adelic}. Theorem~\ref{tm:equidistrib} gives a general criterion to have such an equidistribution statement. What is important here is that a good height function is not necessarily induced by a metrization on a projective model of the variety, or given by an adelic datum. 

\medskip

The strategy of the proof of Theorem~\ref{tm:equidistrib} follows more or less that of Yuan's result. Let us now quickly sketch the proof. Fix an integer $n$ and let $\varphi$ be a test function at place $v$ with compact support in $V_v^\mathrm{an}$ and endow the trivial bundle of $X_n$ with the metric induced by $\varphi_n:=\varphi\circ\psi_n$ at place $v$ and with the trivial metric at all places $w\neq v$. As it is classical, we first use the adelic Minkowski's second Theorem to compare the $\liminf_{i} h_{\bar{L}_n(\varphi)}(F_i)$
with the arithmetic volume $\widehat{\mathrm{vol}}_\chi(\bar{L}_n(\varphi))$ of of the metrized $\mathbb{Q}$-line bundle $\bar{L}_n(\varphi)$. 
Then, we rely on the arithmetic version of Siu's bigness criterion proved by Yuan~\cite{yuan} to get a lower bound on the expansion of $\widehat{\mathrm{vol}}_\chi(\bar{L}_n(t\varphi))$ with respect to $t>0$ of the form $t\langle \mathrm{vol}(L_n)^{-1}c_1(\bar{L}_n)^k_v,\varphi_n\rangle+C_n(\varphi,t)$. When $\dim V>1$, this is ensured by assumption (4).

\medskip

Our main input here is to get an explicit control of the term $C_n(\varphi,t)$ in terms of $\mathrm{vol}(L_n)$, $\sup_{V_v^\mathrm{an}}|\varphi|$ only, for all $t\in(0,1]$. This allows us to find $C\geq1$ depending only on $\varphi$ such that for all $t\in(0,1]$, we have
\[\limsup_{i\to\infty}\left|\langle \mu_{F_i,v},\varphi\rangle-\left\langle \frac{c_1(\bar{L}_n)^k_v}{\mathrm{vol}(L_n)},\varphi_n\right\rangle\right|\leq \frac{\varepsilon_n(\{F_i\}_i)}{t}+Ct.\]
The hypothesis on $h$-small sequences and the assumption that $\mathrm{vol}(L_n)^{-1}(\psi_n)_*c_1(\bar{L}_n)^k_v$ converges to $\mu_v$ allow us to conclude.

\subsection*{Applications in families of dynamical systems}
Our motivation for proving Theorem~\ref{tm:equidistrib} comes from the study of families of dynamical systems.  More precisely, we want to apply Theorem~\ref{tm:equidistrib} to families of polarized endomorphisms.
Let $S$ be a smooth quasi-projective variety of dimension $p\geq1$ and let $\pi:\mathcal{X}\to S$ be a family of smooth projective varieties. We say $(\mathcal{X},f,S)$ is a family of \emph{polarized endomorphisms} if $f:\mathcal{X}\to\mathcal{X}$ is a morphism with $\pi\circ f=\pi$ and if there is a relatively ample line bundle $\mathcal{L}$ on $\mathcal{X}$ and an integer $d\geq2$ such that $f^*\mathcal{L}\simeq \mathcal{L}^{\otimes d}$.
Given a family $(\mathcal{X},f,\mathcal{L})$ and a collection $\mathfrak{a}:=(a_1,\ldots,a_q)$ of sections $a_j:S\to\mathcal{X}$ of $\pi$ are all defined over $\bar{\mathbb{Q}}$, we can define a height function on the variety $S$ by letting
\[h_{f,\mathfrak{a}}(t):=\sum_{j=1}^q\widehat{h}_{f_t}(a_j(t))), \quad t\in S(\bar{\mathbb{Q}}),\]
where $f_t$ is the restriction of $f$ to the fiber $X_t$ of $\pi:\mathcal{X}\to S$ and $\widehat{h}_{f_t}$ is the canonical height of the endomorphism $f_t:X_t\to X_t$, as defined by Call-Silverman~\cite{CS-height}.
If $v\in M_\mathbb{K}$ is archimedean, we can also define a \emph{bifurcation current} $T_{f,a_i}$ on $S_v^\mathrm{an}$ for each dynamical pair $(\mathcal{X},f,\mathcal{L},a_i)$ by letting
\[T_{f,a_i}:=\pi_*\left(\widehat{T}_f\wedge[a_i(S_v^\mathrm{an})]\right).\]
This is a closed positive $(1,1)$-current with continuous potential, see \S~\ref{sec:families} for more details. 


\medskip

As an application of Theorem~\ref{tm:equidistrib}, we prove the following.

\begin{theorem}\label{tm:principal}
Let $(\mathcal{X},f,\mathcal{L})$ be a family of polarized endomorphisms parametrized by a smooth quasiprojective variety $S$ and let $\mathfrak{a}:=(a_1,\ldots,a_q)$ be a collection of sections of $\pi:\mathcal{X}\to S$, all defined over a number field $\mathbb{K}$. Assume the following properties hold:
\begin{enumerate}
\item there exists $v\in M_\mathbb{K}$ archimedean such that the measure 
\[\mu_{f,\mathfrak{a}}:=(T_{f,a_1}+\cdots+T_{f,a_q})^{\dim S}\]
is non-zero with mass $\mathrm{vol}_{f}(\mathfrak{a}):=\mu_{f,a}(S_v^\mathrm{an})>0$,
\item there is a generic $h_{f,\mathfrak{a}}$-small sequence $\{F_i\}_i$ of finite Galois-invariant subsets of $S(\bar{\mathbb{Q}})$.
\end{enumerate}
Pick any $v\in M_\mathbb{K}$ and let $(F_m)_m$ be a $h$-small sequence of Galois-invariant finite subsets of $S(\bar{\mathbb{Q}})$ such that for any hypersurface $H\subset S$ defined over $\mathbb{K}$, we have
\[\# (F_n\cap H)=o(\# F_n), \quad \text{as} \ n\to+\infty.\]
Then the sequence  $(\mu_{F_m,v})_m$ of probability measure on $S_v^\mathrm{an}$ converges to the probability measure $\frac{1}{\mathrm{vol}_{f}(\mathfrak{a})}\mu_{f,a,v}$ in the weak sense of measures on $S_v^\mathrm{an}$.
\end{theorem}

Let us make a few comments on the proof. By Theorem~\ref{tm:equidistrib}, it is sufficient to prove $h_{f,\mathsf{a}}$ is a good height function with associated global measure $\{\mu_{f,\mathfrak{a},v}\}_{v\in M_\mathbb{K}}$. The construction of $(X_n,\bar{L}_n,\psi_n)$ is quite easy and just consists in revisiting the convergence
\[\widehat{h}_{f_t}(a_j(t))=\lim_{n\to\infty}d^{-n}h_{\mathcal{X},\mathcal{L}}(f_t^n(a_j(t))).\]
The convergence of measures $c_1(\bar{L}_n)_v^{\dim S}$ towards $\mu_{f,a,v}$ and the upper bound on the local intersection numbers are also not difficult to establish. The two key facts are the convergence of volumes, which relies on the key estimates of \cite{GV_Northcott}, and the fact that for any small generic sequence $\{F_i\}_i$, the sequence $\varepsilon_n(\{F_i\}_i)$ converges to $0$. To establish this last point, we use a comparison of $\widehat{h}_{f_t}$ with $h_{\mathcal{X},\mathcal{L}}|_{X_t}$ established by Call and Silverman~\cite{CS-height}, as well as the convergence of the volumes and Siu's classical bigness criterion. 
This proof directly inspires from the strategy of \cite{Kuhne}, where the above mentioned arguments replace his use of a deep and recent result on families of abelian varieties due to Gao and Habegger~\cite{Gao-Habegger}.

\begin{center}
$\dag$
\end{center}

To emphasize the strength of Theorem~\ref{tm:principal}, we finish here with a general equiditribution towards the \emph{bifurcation measure} $\mu_\bif$ of the moduli space $\mathcal{M}_d$ as introduced by \cite{BB1}. Recall that the moduli space $\mathcal{M}_d$ of degree $d$ rational maps is the space of $PGL(2)$ conjugacy classes of degree $d$ rational maps, and that it is an irreducible affine variety of dimension $2d-2$ defined over $\mathbb{Q}$,~see~e.g.~\cite{Silverman-Space-rat}.
Recall also that a parameter $\{f\}$ is \emph{post-critically finite} (or PCF) if its post-critical set $\bigcup_{n\geq1}f^{\circ n}(\mathrm{Crit}(f))$ is finite.

\begin{theorem}\label{cor:ratd}
Fix a sequence $(F_n)_n$ of finite subsets of the moduli space $\mathcal{M}_d(\bar{\mathbb{Q}})$ of degree $d$ rational maps of $\mathbb{P}^1$ such that $F_n$ is Galois-invariant for all $n$ and such that, for any hypersurface $H$ of $\mathcal{M}_d$ that is defined over $\mathbb{Q}$, then
\[\# (F_n\cap H)=o(\# F_n), \quad \text{as} \ n\to\infty.\]
Assume that for any $\{f\}\in \bigcup_nF_n$, the map $f$ is PCF. Then the measure $\frac{1}{\# F_n}\sum_{\{f\}\in F_n}\delta_{\{f\}}$ converges weakly to the normalized bifurcation measure $\mu_\mathrm{bif}$.
\end{theorem}

\subsection*{A variant of Theorem~\ref{tm:equidistrib} on projective varieties}
The strength of the proof we give is that it allows to get rid of the existence of a ``height function'' to get an equidistribution result in the spirit of Theorem~\ref{tm:equidistrib}, at least when working on a projective variety.
To highlight this observation, we are now going to give a variant of Theorem~\ref{tm:equidistrib} on \emph{projective} varieties, where the open set over which the model $X_n$ is isomorphic to $X$ actually \emph{depends} on $n$. 

\medskip

Let us now be more precise. We let $X$ be a projective variety of dimension $k$ defined over $\bar{\mathbb{Q}}$ and we fix a place $v\in M_\mathbb{K}$. For any $n\geq0$, we let be a birational morphism $\psi_n:X_n\to X$ and we let $L_n$  be a big and nef $\mathbb{Q}$-line bundle endowed with a semi-positive adelic continuous metrization $\bar{L}_n$. We assume that
\begin{enumerate}
\item the sequence $\mathrm{vol}(L_n)$ converges to  constant $\mathrm{V}>0$ and the sequence of probability measures $(\mathrm{vol}(L_n)^{-1}(\psi_n)_*c_1(\bar{L}_n)^k_v)_n$ converges weakly to a probability measure $\mu_v$ on $X_v^\mathrm{an}$,
\item If $k:=\dim X>1$, for any ample line bundle $M_0$ on $X$ and any adelic semi-positive continuous metrization $\bar{M}_0$ on $M_0$, there is a constant $C\geq0$ such that 
\[\left(\psi_n^*(\bar{M}_0)\right)^j\cdot \left(\bar{L}_n\right)^{k+1-j}\leq C,\]
for any $2\leq j\leq k+1$ and any $n\geq0$.
\end{enumerate}
\begin{definition}
The data $(X,\mu_v,X_n,\bar{L}_n)$ is a \emph{quasi-height} on $X$ at place $v$.
\end{definition}

A sequence $(F_i)_i$ of Galois-invariant finite subsets of $X(\bar{\mathbb{Q}})$ is \emph{quasi-small} if $\psi_n^{-1}\{F_i\}$ is a finite subset of $X_n(\bar{\mathbb{Q}})$ for any $n\geq0$ and any $i$ and if the sequence
\[\varepsilon_n(\{F_i\}_i):=\limsup_ih_{\bar{L}_n}(\psi_n^{-1}(F_i))-h_{\bar{L}_n}(X_n)\] satisfies $\varepsilon_n(\{F_i\})\to0$ as $n\to\infty$.

\bigskip

As in the case of good height functions on a quasi-projective variety, quasi-small points equidistribute the measure $\mu_v$. The precise statement is the following.

\begin{theorem}[Equidistribution of quasi-small points]\label{tm:equidistrib-proj}
Let $X$ be a projective variety defined over a number field $\mathbb{K}$, let $v\in M_\mathbb{K}$ and let $(X,\mu_v,X_n,\bar{L}_n)$ is a quasi-height on $X$ at place $v$. For any quasi-small sequence $(F_m)_m$ of Galois-invariant finite subsets of $X(\bar{\mathbb{Q}})$ such that for any hypersurface $H\subset V$ defined over $\mathbb{K}$, we have
\[\# (F_n\cap H)=o(\# F_n), \quad \text{as} \ n\to+\infty,\]
the probability measure $\mu_{F_m,v}$ on $X_v^\mathrm{an}$ which is  equidistributed on $F_m$ converges to $\mu_v$ in the weak sense of measures, i.e. for any continuous function with compact support $\varphi\in\mathscr{C}^0(X_v^\mathrm{an})$, we have
\[\lim_{m\to\infty}\frac{1}{\# F_m}\sum_{y\in F_m}\varphi(y)=\int_{X_v^\mathrm{an}}\varphi\,\mu_v.\]
\end{theorem}

\par\noindent The proof of this result follows closely that of Theorem~\ref{tm:equidistrib} and we will explain how to adapt the arguments, when needed.

\paragraph*{Organization of the paper}
Section~\ref{sec:ineg-height} is devoted to quantitative height inequalities which are use in the proof of Theorem~\ref{tm:equidistrib}. Section~\ref{sec:proofgoodheight} is dedicated to the proof of Theorem~\ref{tm:equidistrib}, and we prove in Subsection~\ref{sec:qa} that quasi-adelic measures induce good height functions. Theorem~\ref{tm:principal} is proved in Section~\ref{sec:distrib-para}. Finally, in Section~\ref{sec:application} we discuss the assumptions of Theorem~\ref{tm:principal}, proving they are sharp and we prove Theorem~\ref{cor:ratd}.

\paragraph*{Acknowledgment} I heartily thank S\'ebastien Boucksom, Charles Favre and Gabriel Vigny. This paper benefited from invaluable comments from them. I also want to thank Xinyi Yuan for very interesting discussions and for pointing out an error in an earlier version of Proposition~\ref{prop:almost-Zhang-quantitative}. I also want to thank Laura DeMarco and Lars K\"uhne for interesting discussions. I finally would like to thank Yuan and Zhang for kindly sharing their draft.  Finally, I thank the Ecole Polytechnique where I was working unitl september 2021.

\section{Quantitative height inequalities}\label{sec:ineg-height}
In the whole section, we let $X$ be a smooth projective variety of dimension $k$ defined over a number field $\mathbb{K}$. 
\subsection{Arithmetic intersection and heights}
For the material of this section, we refer to \cite{ACL2} and \cite{Zhang-positivity}. Let $L_0,\ldots,L_k$ be $\mathbb{Q}$-line bundle on $X$. Assume $L_i$ is equipped with an adelic continuous metric $\{\|\cdot\|_{v,i}\}_{v\in M_\mathbb{K}}$ and we denote $\bar{L}_i:=(L_i,\{\|\cdot\|_v\}_{v\in M_\mathbb{K}})$. Assume $\bar{L}_i$ is  semi-positive for $1\leq i\leq k$ and $\bar{L}_0$ is integrable, i.e. can be written as a difference of semi-positive adelic line bundles.

~

Fix a place $v\in M_\mathbb{K}$. Denote by $X_v^\mathrm{an}$ the Berkovich analytification of $X$ at the place $v$. We also let $c_1(\bar{L}_i)_v$ be the curvature form of the metric $\|\cdot\|_{v,i}$ on $X_v^\mathrm{an}$.
We will use in the sequel that the arithmetic intersection number $\left(\bar{L}_0\cdots\bar{L}_k\right)$ is symmetric and multilinear with respect to the $L_i$ and that
\[\left(\bar{L}_0\right)\cdots\left(\bar{L}_k\right)=\left(\bar{L}_1|_{\mathrm{div}(s)}\right)\cdots\left(\bar{L}_k|_{\mathrm{div}(s)}\right)+\sum_{v\in M_\mathbb{K}}\int_{X_v^{\mathrm{an}}}\log\|s\|^{-1}_v \bigwedge_{j=1}^kc_1(\bar{L}_i)_v,\]
for any global section $s\in H^0(X,L_0)$ (whenever such a section exists). In particular, if $L_0$ is the trivial bundle and $\|\cdot\|_{v,0}$ is the trivial metric at all places but $v_0$, this gives
\[\left(\bar{L}_0\right)\cdots\left(\bar{L}_k\right)=\int_{X_{v_0}^{\mathrm{an}}}\log\|1\|^{-1}_{v_0,0} \bigwedge_{j=1}^kc_1(\bar{L}_i)_{v_0}.\]
When $\bar{L}$ is a big and nef $\mathbb{Q}$-line bundle endowed with a semi-positive continuous adelic metric, following Zhang~\cite{Zhang-positivity}, we define $h_{\bar{L}}(X)$ as
\[ h_{\bar{L}}(X):=\frac{\left(\bar{L}\right)^{k+1}}{(k+1)[\mathbb{K}:\mathbb{Q}]\mathrm{vol}(L)},\]
where $\mathrm{vol}(L)=(L)^k$ is the volume of the line bundle $L$ (also denoted by $\deg_X(L)$ sometimes).
We also define the height of a closed point $x\in X(\bar{\mathbb{Q}})$ as 
\[h_{\bar{L}}(x)=\frac{\left(\bar{L}|x\right)}{[\mathbb{K}:\mathbb{Q}]}=\frac{1}{[\mathbb{K}:\mathbb{Q}]\#\mathsf{O}(x)}\sum_{v\in M_\mathbb{K}}\sum_{\sigma:\mathbb{K}(x)\hookrightarrow \C_v}\log\|s(\sigma(x))\|_{v}^{-1},\]
where $\mathsf{O}(x)$ is the Galois orbit of $x$, for any section $s\in H^0(X,L)$ which does not vanish at $x$. Finally, for any Galois-invariant finite set $F\subset X(\bar{\mathbb{Q}})$, we define $h_{\bar{L}}(F)$ as
\[h_{\bar{L}}(F):=\frac{1}{\# F}\sum_{y\in F}h_{\bar{L}}(y).\]
A fundamental estimate is the following, called Zhang's inequalities~\cite{Zhang-positivity}:
\begin{lemma}[Zhang]\label{ineg-Zhang-below}
If $L$ is ample and $e(\bar{L})=\sup_H\inf_{x\in (X\setminus H)(\bar{\mathbb{Q}})}h_{\bar{L}}(x)$, where the supremum is taken over all hypersurfaces $H$ of $X$ defined over $\mathbb{K}$, then
\[\frac{1}{k+1}\left(e(\bar{L})+k\inf_{y\in X(\bar{\mathbb{Q}})}h_{\bar{L}}(y)\right)\leq h_{\bar{L}}(X)\leq e(\bar{L}).\]
In particular, if $h_{\bar{L}}(x)\geq0$ for all $x\in X(\bar{\mathbb{Q}})$, then $h_{\bar{L}}(X)\geq0$.
\end{lemma}

In particular, we can deduce the following.
\begin{corollary}\label{lm:Zhang}
Assume $L$ is big and nef $\mathbb{Q}$-line bundle on $X$ which is endowed with an adelic semi-positive continuous metrization $\{\|\cdot\|_v\}_{v\in M_\mathbb{K}}$. Assume $h_{\bar{L}}(x)\geq0$ for any point $x\in X(\bar{\mathbb{Q}})$, then $h_{\bar{L}}(X)\geq0$.
\end{corollary}

\begin{proof}
Let $\bar{A}$ be an ample hermitian line bundle on $X$ with $h_{\bar{A}}\geq0$ on $X$ and $\epsilon>0$ rational. Then $\bar{L}(\epsilon):=\bar{L}+\epsilon \bar{A}$ is an ample adelically metrized $\mathbb{Q}$-line bundle and $h_{\bar{L}(\epsilon)}\geq0$ on $X(\bar{\mathbb{Q}})$. By Lemma~\ref{ineg-Zhang-below}, we have $h_{\bar{L}(\epsilon)}(X)\geq0$. First, note that
\[\mathrm{vol}(L(\epsilon))=\left(L+\epsilon A\right)^k=\sum_{j=0}^k {k \choose j}\left(L\right)^j\left(\epsilon A\right)^{k-j}=\mathrm{vol}(L)+O(\epsilon).\]
Also, we can compute similarly
\begin{align*}
\left(\bar{L}(\epsilon)\right)^{k+1}=\sum_{j=0}^{k+1} {k+1 \choose j} \epsilon^j\left(\bar{L}\right)^{k+1-j}\left(\bar{A}\right)^{j}=\left(\bar{L}\right)^{k+1}+O(\epsilon).
\end{align*}
We thus deduce that
\[0\leq h_{\bar{L}(\epsilon)}(X)=h_{\bar{L}}(X)+O(\epsilon)\]
and the conclusion follows making $\epsilon\to0$.
\end{proof}

\subsection{Test functions}\label{sec:test}

Let $L$ be a big and nef $\mathbb{Q}$-line bundle on $X$. We equip $L$ with an adelic continuous semi-positive metric $\{\|\cdot\|_v\}_{v\in M_\mathbb{K}}$ and we denote $\bar{L}:=(L,\{\|\cdot\|_v\}_{v\in M_\mathbb{K}})$. We also fix a Zariski open set $V\subset X$.
Let $v\in M_\mathbb{K}$ be any place of the field $\mathbb{K}$.

\begin{definition}
A function $\varphi:X_{v}^\mathrm{an}\to\mathbb{R}$ is called a \emph{test function} on $X_{v}^\mathrm{an}$ if it is
\begin{itemize}
\item a $\mathscr{C}^\infty$-smooth function on $X_v^\mathrm{an}$ if $v$ is archimedean,
\item a $\mathbb{Q}$-model function on $X_v^\mathrm{an}$ if $v$ is non-archimedean.
\end{itemize}
We also say $\varphi$ is a \emph{test function on} $V_v^\mathrm{an}$ is it is a test function on $X_v^\mathrm{an}$ and if its support is compactly contained in $V_v^\mathrm{an}$.
\end{definition}

 If $v$ is archimedean, as the space $\mathscr{C}^\infty_c(V_v^\mathrm{an})$ of smooth functions on $V_v^\mathrm{an}$ with compact support is dense in the space $\mathscr{C}^0_c(V_v^\mathrm{an})$ of continuous functions  on $V_v^\mathrm{an}$ with compact support, it is sufficient to test the convergence against smooth functions.
If $v$ is non-archimedean, we use a result of Gubler~\cite[Theorem~7.12]{Gubler-local}: $\mathbb{Q}$-model functions with compact support in $V_v^\mathrm{an}$ are dense in the space $\mathscr{C}^0_c(V_v^\mathrm{an})$ of continuous functions  on $V_v^\mathrm{an}$ with compact support, it is sufficient to test the convergence against compactly supported model functions. In particular, we have the following:

\begin{lemma}[Gubler]\label{lm:Gubler}
Test functions are dense in the space $\mathscr{C}^0_c(V_v^\mathrm{an})$ (resp. in $\mathscr{C}^0(X_v^\mathrm{an})$) of continuous functions on $V_v^\mathrm{an}$ with compact support (resp. of continuous functions on $X_v^\mathrm{an}$).
\end{lemma}

%
%
%

\subsection{Arithmetic volume and the adelic Minkowski Theorem}
We want here to use the arithmetic Minkowski's second Theorem to relate the limit of $h_{\bar{M}}(F_i)$ along a generic sequence of finite subsets $\{F_i\}_i$ of $X(\bar{\mathbb{Q}})$ and the arithmetic volume of $\bar{M}$, for any continuous adelic metrized $\mathbb{Q}$-line bundle. For any $v\in M_\mathbb{K}$, and any section $s\in H^0(X,M)$, set $\|s\|_{v,\sup}:=\sup_{x\in X(\bar{\mathbb{Q}}_v)}\|s(x)\|_v$. 
Assume $H^0(X,M)\neq0$. If $\mathbb{A}_\mathbb{K}$ is the ring of adels of $\mathbb{K}$ and if $\mu$ is a Haar measure on the locally compact abelian group  $H^0(X,M)\otimes\mathbb{A}_\mathbb{K}$, let
\[\chi_{\sup}(\bar{M})=-\log\frac{\mu(H^0(X,M)\otimes\mathbb{A}_\mathbb{K}/H^0(X,M))}{\mu(\prod_vB_v)},\]
where $B_v$ is the (closed) unit ball of $H^0(X,M)\otimes \mathbb{K}_v$ for the norm induced by $\|\cdot\|_{v,\sup}$, see~e.g.~\cite{CL-T} or \cite{Zhang-small}. The \emph{arithmetic volume} $\widehat{\mathrm{vol}}_\chi(\bar{M})$ of $\bar{M}$ is the defined as
\[\widehat{\mathrm{vol}}_\chi(\bar{M}):=\limsup_{N\to\infty}\frac{\chi_{\sup}(N\bar{M})}{N^{k+1}/(k+1)!}.\]

\medskip

Choose a place $v_0\in M_\mathbb{K}$ and let $\varphi:X_{v_0}^\mathrm{an}\to\mathbb{R}$ be a test functions, i.e. $\varphi\in\mathscr{C}^\infty(X_{v_0}^\mathrm{an})$ if $v_0$ is archimedean and $\varphi$ is a model function otherwise. We define a twisted metrized line bundle $\bar{M}(\varphi)$ by changing the metric at the place $v_0$ as follows: let $\|\cdot\|_{v_0,\varphi}:=\|\cdot\|_{v_0}e^{-\varphi}$.

\medskip

Recall that a sequence $(F_i)_i$ of Galois-invariant subsets of $X(\bar{\mathbb{Q}})$ is \emph{generic} if for any hypersurface $Z$ of $X$ defined over $\mathbb{K}$, there is $i_0\geq0$ such that $F_i\cap Z(\bar{\mathbb{Q}})=\varnothing$ for all $i\geq i_0$. Using the arithmetic Minkowski's second Theorem, we can prove the next lemma using a classical argument (see, e.g.,~\cite{Berman-Boucksom}).
\begin{lemma}\label{lm:Minkowski}
For any big and nef line bundle $\bar{M}$ endowed with an adelic semi-positive continuous metrization, any $v_0\in M_\mathbb{K}$, any test function $\varphi:X_{v_0}^{\mathrm{an}}\to\mathbb{R}$ and any generic sequence $(F_i)_i$ of Galois-invariant subsets of $X(\bar{\mathbb{Q}})$, we have
\[\liminf_{i\to\infty}h_{\bar{M}(\varphi)}(F_i)\geq \frac{\widehat{\mathrm{vol}}_\chi(\bar{M}(\varphi))}{[\mathbb{K}:\mathbb{Q}](k+1)\mathrm{vol}(M)}.\]
\end{lemma}

\begin{proof}
Taking $c$ large enough, we have 
\[\widehat{\mathrm{vol}}_\chi(\bar{M}(\varphi+c))=\widehat{\mathrm{vol}}_\chi(\bar{M}(\varphi))+c(k+1)\mathrm{vol}(M)>0.\]
By Minkowski's theorem~\cite[Theorem~C.2.11]{bombieri-gubler}, as soon as $\widehat{\mathrm{vol}}_\chi(\bar{M}(\varphi+c))>0$, there exists a non-zero small section $s\in H^0(X, N L)$ such that
\begin{align*}
\log\|s\|_{v_0} &\leq -\frac{\widehat{\mathrm{vol}}_\chi(\bar{M}(\varphi))}{[\mathbb{K}:\mathbb{Q}](k+1)\mathrm{vol}(M)}N+o(N)
\end{align*}
and $\log\|s\|_v\leq 0$ for any other place $v\in M_\mathbb{K}$, where $\|\cdot\|_v$ is the metrization of $N\bar{M}(\varphi+c)$ at the place $v$ induced by that of $\bar{M}$. As $(F_i)_i$ is generic, there is $i_0$ such that $F_i\cap\,\mathrm{supp}(\mathrm{div}(s))=\varnothing$ for $i\geq i_0$. We thus can compute the height $h_{N\bar{M}(\varphi+c)}$ of $F_i$ using this small section $s$ to find
\begin{align*}
h_{\bar{M}(\varphi+c)}(F_i)=\frac{1}{N}h_{N\bar{M}(\varphi+c)}(F_i)&\geq\frac{\widehat{\mathrm{vol}}_\chi(\bar{M}(\varphi+c))}{[\mathbb{K}:\mathbb{Q}](k+1)\mathrm{vol}(M)}+o(1).
\end{align*}
Making $N\to\infty$, this gives
\begin{align}
h_{\bar{M}(\varphi+c)}(F_i)\geq\frac{\widehat{\mathrm{vol}}_\chi(\bar{M}(\varphi+c))}{[\mathbb{K}:\mathbb{Q}](k+1)}.\label{ineg-vol-arithmample}
\end{align}
To conclude, we use $\widehat{\mathrm{vol}}_\chi(\bar{M}(c+\varphi))=\widehat{\mathrm{vol}}_\chi(\bar{M}(\varphi))+c(k+1)\mathrm{vol}(M)$
and we remark that by definition,  for any $c>0$ and any closed point $x\in X(\bar{\mathbb{Q}})$,
\[h_{\bar{M}(c+\varphi)}(x)=h_{\bar{M}(\varphi)}(x)+\frac{c}{[\mathbb{K}:\mathbb{Q}]},\]
so that the inequality~\eqref{ineg-vol-arithmample} is the expected result.
\end{proof}

\subsection{A lower bound on the arithmetic volume}

Let now $L$ be a big and nef $\mathbb{Q}$-line bundle on $X$. We equip $L$ with an adelic continuous semi-positive metric $\{\|\cdot\|_v\}_{v\in M_\mathbb{K}}$ and we denote $\bar{L}:=(L,\{\|\cdot\|_v\}_{v\in M_\mathbb{K}})$. We also fix a Zariski open set $V\subset X$. 
\medskip

Fix a place $v\in M_\mathbb{K}$ and pick any test function $\varphi:X_{v}^\mathrm{an}\to\mathbb{R}$ with compact support in $V_{v}^\mathrm{an}$. Let $\bar{\mathcal{O}}_{v}(\varphi)$ the trivial line bundle on $X$ equipped with the trivial metric at all places but $v$ and equipped with the metric induced by $\varphi$ at the place $v$. For any $t\in\mathbb{R}$, we let
\[\bar{L}(t\varphi):=\bar{L}+t\bar{\mathcal{O}}_{v}(\varphi).\]  The key fact of this section is the next proposition which relies on Yuan's arithmetic bigness criterion \`a la Siu~\cite[Theorem~2.2]{yuan} :

\begin{proposition}\label{prop:almost-Zhang-quantitative}
Let $M$ be a big and nef line bundle on $X$. For any non-constant test function $\varphi:X_{v}^\mathrm{an}\to\mathbb{R}$, and any decomposition $\bar{\mathcal{O}}(\varphi)=\bar{M}_+-\bar{M}_-$ as a difference of big and nef adelic metrized line bundle with underlying line bundle $M$ and any $t>0$, we have
\begin{align*}
\frac{\widehat{\mathrm{vol}}_\chi(\bar{L}(t\varphi))}{[\mathbb{K}:\mathbb{Q}](k+1)\mathrm{vol}(L)}\geq h_{\bar{L}}(X) +&\frac{t}{[\mathbb{K}:\mathbb{Q}]}\int_{X_v^\mathrm{an}} \varphi\,\frac{c_1(\bar{L})_v^k}{\mathrm{vol}(L)}\\
&-t\frac{\sup_{X_v^\mathrm{an}}|\varphi|}{[\mathbb{K}:\mathbb{Q}]}\Bigg(\frac{\mathrm{vol}(L+tM)}{\mathrm{vol}(L)}-1\Bigg)\\
&-\sum_{j=2}^{k+1}\frac{(k+2-j)}{[\mathbb{K}:\mathbb{Q}]\mathrm{vol}(L)} {k+1 \choose j}\left(\bar{L}\right)^{k-j+1}\cdot \left(t\bar{M}_+\right)^{j}.
\end{align*}
Moreover, when $\varphi$ is compactly supported in $V_{v}^\mathrm{an}$, $\sup_{X_v^\mathrm{an}}|\varphi|$ can be replaced by $\sup_{V_v^\mathrm{an}}|\varphi|$ and the integral can be computed on $V_v^\mathrm{an}$.
\end{proposition}

\begin{proof}
Write $\varphi=\psi_+-\psi_-$, where $\psi_\pm$ is a smooth psh metric on $M_{v}^\mathrm{an}$. On then can write $\bar{\mathcal{O}}_{v}(t\varphi)=t\bar{M}_+-t\bar{M}_-$, where $\bar{M}_\pm$ are the induced metrizations on $M$. Extend both $t\bar{M}_\pm$ as adelic metrized line bundles which coincide at all places $w\neq v$ and pick any ample arithmetic line bundle $\bar{L}_0$ on $X$. As $M$ is nef, for a given $q\geq1$, $q\bar{M}_\pm+\bar{L}_0$ is an ample hermitian line bundle.
The line bundle $q\bar{L}(t\varphi)$ is then the difference of two ample hermitian line bundles:
\[q\bar{L}(t\varphi)=\left(q\bar{L}+tq\bar{M}_+\right)-tq\bar{M}_-=\left(q\bar{L}+t\left(q\bar{M}_++\bar{L}_0\right)\right)-t\left(q\bar{M}_-+\bar{L}_0\right).\]
Apply Yuan's arithmetic bigness criterion~\cite[Theorem~2.2]{yuan} to multiples of $N\bar{L}(t\varphi)$, where $q$ divides $N$ and making $N\to\infty$ gives
\begin{align}
\widehat{\mathrm{vol}}_\chi(\bar{L}(t\varphi))
\geq \left(\bar{L}+t\bar{M}_+\right)^{k+1}-(k+1)\left(\bar{L}+t\bar{M}_+\right)^{k}\cdot \left(t\bar{M}_-\right)\label{ineq:yuan-big}
\end{align}
 Remark that $\left(\bar{L}+t\bar{M}_+\right)^{k+1}-(k+1)\left(\bar{L}+t\bar{M}_+\right)^{k}\cdot \left(t\bar{M}_-
\right)$ expands as
\begin{align*}
\left(\bar{L}\right)^{k+1}+& \sum_{j=1}^{k+1} {k+1 \choose j}\left(\bar{\mathcal{O}}_v(\varphi)\right)\cdot\left(\bar{L}\right)^{k-j+1}\cdot \left(t\bar{M}_+\right)^{j-1}\\
&-\sum_{j=2}^{k+1}(k-j){k+1 \choose j}\left(\bar{L}\right)^{k-j+1}\cdot \left(t\bar{M}_+\right)^{j},
\end{align*}
where we used that $t\bar{M}_+-t\bar{M}_-=\bar{\mathcal{O}}_v(t\varphi)$. As the underlying line bundle of $\bar{\mathcal{O}}_v(\varphi)$ is trivial, we can compute this intersection number with the use of the constant section $1$. The term $j=1$ is exactly the integral 
\[\left(\bar{\mathcal{O}}_v(t\varphi)\right)\cdot\left(\bar{L}\right)^{k}=t\int_{X_v^\mathrm{an}}\varphi\, c_1(\bar{L})_v^k\]
and for $j\geq2$, we can compute
\begin{align*}
\left(\bar{\mathcal{O}}_v(t\varphi)\right)\cdot\left(\bar{L}\right)^{k-j+1}\cdot \left(t\bar{M}_+\right)^{j-1} & =t\int_{X_v^\mathrm{an}}\varphi\, c_1(\bar{L})_v^{k-j+1}\wedge \left(tc_1(M,\psi_+)_v\right)^{j-1}\\
&\geq -t\sup_{X_v^\mathrm{an}}|\varphi|\int_{X_v^\mathrm{an}}c_1(\bar{L})_v^{k-j+1}\wedge \left(tc_1(M,\psi_+)_v\right)^{j-1}\\
&\geq -t\sup_{X_v^\mathrm{an}}|\varphi|(L^{k+1-j}\cdot (tM)^{j-1}).
\end{align*}
Summing from $j=2$ to $k+1$, we find
\begin{align*}
\sum_{j=2}^{k+1} {k+1 \choose j}\left(\bar{\mathcal{O}}_v(\varphi)\right)\cdot\left(\bar{L}\right)^{k-j+1}\cdot\left(t\bar{M}_+\right)^{j} & \geq -t\sup_{X_v^\mathrm{an}}|\varphi|\sum_{j=2}^{k+1} {k+1 \choose j} (L^{k+1-j}\cdot (tM)^{j-1})\\
&\geq -t\sup_{X_v^\mathrm{an}}|\varphi|\sum_{j=1}^{k} \frac{k+1}{j+1} {k \choose j} (L^{k-j}\cdot (tM)^{j})\\
&\geq -t(k+1)\sup_{X_v^\mathrm{an}}|\varphi|\left(\mathrm{vol}(L+tM)-\mathrm{vol}(L)\right).
\end{align*}
The above summarizes as
\begin{align*}
\widehat{\mathrm{vol}}_\chi(\bar{L}(t\varphi))\geq \left(\bar{L}\right)^{k+1}+(k+1)t & \int_{X_v^\mathrm{an}}\varphi\, c_1(\bar{L})_v^k\\
&-t(k+1)\sup_{X_v^\mathrm{an}}|\varphi|\left(\mathrm{vol}(L+tM)-\mathrm{vol}(L)\right)\\
& -\sum_{j=2}^{k+1}(k-j){k+1 \choose j}\left(\bar{L}\right)^{k-j+1}\cdot \left(t\bar{M}_+\right)^{j}.
\end{align*}
The conclusion follows dividing by $(k+1)[\mathbb{K}:\mathbb{Q}]\mathrm{vol}(L)$.

When $\varphi$ is compactly supported in $V_v^\mathrm{an}$, it is obvious that the sup and the integral can be taken over $V_v^\mathrm{an}$. 
\end{proof}

\section{Good height functions and equidistribution}\label{sec:proofgoodheight}

In this section, we prove Theorem~\ref{tm:equidistrib} using the estimates established in Section~\ref{sec:ineg-height}. Let $\mathbb{K}$ be z number field, let $V$ be a smooth quasi-projective variety of dimension $k$ defined over $\mathbb{K}$. 
Let $h$ be a $v$-good height function on $V$ with induced measure $\mu_v$ on $V_v^\mathrm{an}$ and $\mathrm{vol}(h)>0$ be its \emph{volume}. Let $(X_n,\bar{L}_n,\psi_n)$ be given by the definition of good height function.

\subsection{A preliminary lemma}\label{sec:prelim}

Let $\iota_0:X_0\hookrightarrow \mathbb{P}^N$ be an embedding defined over $\mathbb{K}$ and let $M_0$ be an ample line bundle on $X_0$. For $n\geq0$, let
\[\bar{M}_n:=\psi_n^*(\bar{M}_0).\]
Since $\psi_n$ is a birational morphism and since $M_0$ is ample, $M_n$ is big and nef.


An important ingredient is the following:

\begin{lemma}\label{lm:bounded-volume}
The sequence $\mathrm{vol}(L_n+M_n)$ converges to a limit $\ell>0$. In particular, there exists a constant $C\geq1$ such that $0\leq \mathrm{vol}(L_n+M_n)\leq C$, for all $n\geq0$.
\end{lemma}
\begin{proof}
By definition of the line bundle $M_n$, $M_n$ is big and nef, so that $\mathrm{vol}(M_n)=(M_n)^k=(M_0)^k=\mathrm{vol}(M_0)$, hence it is independent of $n$. 
By the assumption (2) of the definition of good height function, we also have
\begin{align*}
(M_0)^k=\frac{(M_0)^k}{\mathrm{vol}(h)}\mathrm{vol}(L_n)+o(\mathrm{vol}(L_n))
& =\frac{\mathrm{vol}(M_0)}{\mathrm{vol}(h)}\mathrm{vol}(L_n)+o(\mathrm{vol}(L_n)).
\end{align*}
In particular, there exists a constant $C_1\geq1$ independent of $n$ such that $\mathrm{vol}(M_n)\leq C_1\mathrm{vol}(L_n)$ and thus
\[\mathrm{vol}(L_n+M_n)=(L_n+M_n)^k\leq (1+C_1)^k(L_n)^k= (1+C_1)^k\mathrm{vol}(L_n).\]
As $\mathrm{vol}(L_n)$  is bounded, the conclusion follows. 
\end{proof}

\subsection{Proof of Theorem~\ref{tm:equidistrib} and of Theorem~\ref{tm:equidistrib-proj}}
We give here the proof of Theorem~\ref{tm:equidistrib} and we explain where to adapt the arguments to prove Theorem~\ref{tm:equidistrib-proj}.
Let $v\in M_\mathbb{K}$ be such that $h$ is $v$-good and let $\varphi:V_v^\mathrm{an}\to\mathbb{R}$ be a test function on $V_v^\mathrm{an}$.
For $t>0$, we define an adelic metrized line bundle by
\[\bar{L}_n(t\varphi_n):=\bar{L}_n+\bar{\mathcal{O}}(t\varphi_n).\]
Let us first prove the result for a sequence $(F_i)_i$ of generic and $h$-small Galois-invariant finite subsets $F_i\subset V(\bar{\mathbb{Q}})$ and let $F_{i,n}:=\psi_n^{-1}(F_i)$. (Here we assume $\{F_i\}_i$ is quasi-small in the quasi-height case).

 We apply Proposition~\ref{prop:almost-Zhang-quantitative} and Lemma~\ref{lm:Minkowski} to find
\begin{align*}
\liminf_{i\to\infty}h_{\bar{L}_n(t\varphi_n)}(F_{i,n})  \geq  h_{\bar{L}_n}(X_n) & -\frac{t}{[\mathbb{K}:\mathbb{Q}]}\int_{V_v^\mathrm{an}}\varphi_n\frac{c_1(\bar{L}_n)_v^k}{\mathrm{vol}(L_n)}\\
&-t^2\frac{\sup_{V_v^\mathrm{an}}|\varphi|}{[\mathbb{K}:\mathbb{Q}]}\Bigg(\frac{\mathrm{vol}(L_n+M_n)}{\mathrm{vol}(L_n)}-1\Bigg)\\
& -\sum_{j=2}^{k+1}\frac{(k-j)}{[\mathbb{K}:\mathbb{Q}]\mathrm{vol}(L)} {k+1 \choose j}\left(\bar{L}_n\right)^{k-j+1}\cdot \left(t\bar{M}_{n,+}\right)^{j}
\end{align*}
for any decomposition $\bar{\mathcal{O}}_v(\varphi_n)=\bar{M}_{n,+}-\bar{M}_{n,-}$, where $\bar{M}_{n,\pm}$ is a semi-positive continuous metrized big and nef line bundle on $(X_n)_v^\mathrm{an}$ with underlying line bundle $M_n$, as defined in section~\ref{sec:prelim}. If we write $\varphi=\psi_+-\psi_-$ where $\psi_\pm$ are metrizations on $M_{0,v}^\mathrm{an}$, then one can obviously write $\varphi_{n}=\psi_+\circ \psi_n-\psi_-\circ \psi_n$. We thus have $\bar{M}_{n,+}=\psi_n^*(\bar{M}_{0,+})$.

As $\mathrm{vol}(L_n)\to\mathrm{vol}(h)>0$, by  Lemma~\ref{lm:bounded-volume}, there exists $C_1\geq1$ independent of $n$ such that
\[\frac{\mathrm{vol}(L_n+M_n)}{\mathrm{vol}(L_n)}\leq C_1,\]
for any $n\geq0$. Using hypothesis (4), we deduce there is $C_2\geq1$ independent of $n$ such that for any $0<t<1$ we find
\begin{align*}
\liminf_{i\to\infty}h_{\bar{L}_n(t\varphi_n)}(F_{i,n}) & \geq h_{\bar{L}_n}(X_n)-\frac{t}{[\mathbb{K}:\mathbb{Q}]}\int_{(X_n)_v^\mathrm{an}}\varphi_n\frac{c_1(\bar{L}_n)_v^k}{\mathrm{vol}(L_n)}-C_2t^2.
\end{align*}
By definition of $\bar{L}_n(t\varphi_n)$, we can compute
\begin{align*}
h_{\bar{L}_n(t\varphi_n)}(F_{i,n})& =h_{\bar{L}_n}(F_{i,n})+\frac{t}{[\mathbb{K}:\mathbb{Q}]\# F_i}\sum_{y\in F_i}\varphi(y).
\end{align*}
Using that the sequence $(F_i)_i$ is $h$-small and generic, assumption 1. gives
\begin{align*}
\liminf_{i\to\infty}h_{\bar{L}_n(t)}(F_{i,n})\leq \varepsilon_n(\{F_i\}_i)+h_{\bar{L}_n}(X_n)+\liminf_{i\to+\infty}\frac{t}{[\mathbb{K}:\mathbb{Q}]F_i}\sum_{y\in F_i}\varphi(y),
\end{align*}
for $0<t<1$. Combined with the above and divided by $t/[\mathbb{K}:\mathbb{Q}]$, this gives
\begin{align*}
\liminf_{i\to+\infty}\left(\frac{1}{\# F_i}\sum_{y\in F_i}\varphi(y)-\int_{(X_n)_v^\mathrm{an}}\varphi_n\,\frac{c_1(\bar{L}_n)_v^{k}}{\mathrm{vol}(L_n)}\right)\geq-\frac{\varepsilon_n(\{F_i\}_i)}{t}-C_2t
\end{align*}
for all $0<t<1$.
Replacing $\varphi$ by $-\varphi$ gives the converse inequality, whence
\begin{align*}
\limsup_{i\to+\infty}\left|\frac{1}{\# F_i}\sum_{y\in F_i}\varphi(y)-\int_{(X_n)_v^\mathrm{an}}\varphi_n\,\frac{c_1(\bar{L}_n)_v^{k}}{\mathrm{vol}(L_n)}\right|\leq\frac{\varepsilon_n(\{F_i\}_i)}{t}+C_1t,
\end{align*}
for all $0<t<1$. One can also remark that since $\psi_n$ is an isomorphism over $V$, we have
\begin{align*}
\int_{(X_n)_v^\mathrm{an}}\varphi_n\,\frac{c_1(\bar{L}_n)_v^{k}}{\mathrm{vol}(L_n)} & = \int_{X_v^\mathrm{an}}\varphi\,\frac{(\psi_n)_*c_1(\bar{L}_n)_v^{k}}{\mathrm{vol}(L_n)}.
\end{align*}
In the quasi-height, we use here that the measure $c_1(\bar{L}_n)_v^{k}$ doesn't give mass to Zariski closed subsets of $X_n$, since it is induced by a smooth metrization on $L_n$.

Fix now $\varepsilon>0$ small. We now let $t:= \varepsilon/2C_2>0$. By assumption, we can choose $n_0\geq1$ such that we have $\frac{\varepsilon_n(\{x_i\}_i)}{t}\leq \varepsilon/2$ for any $n\geq n_0$. Therefore, for $n\geq n_0$,
\begin{align}
\limsup_{i\to+\infty}\left|\frac{1}{\# F_i}\sum_{y\in F_i}\varphi(y)-\int_{V_v^\mathrm{an}}\varphi\,\frac{(\psi_n)_*c_1(\bar{L}_n)_v^{k}}{\mathrm{vol}(L_n)}\right|\leq\varepsilon.\label{ineq-almost-there-na}
\end{align}
By assumption $(2)$ of the definition of good height, we can choose $n_0$ such that for any $n\geq n_0$,
\[\left|\int_{X_v^\mathrm{an}}\varphi\, \frac{(\psi_n)_*c_1(\bar{L}_n)_v^{k}}{\mathrm{vol}(L_n)}-\int_{X_v^\mathrm{an}}\varphi\, \mu_v\right|\leq \left(1+\|\varphi\|_{L^{\infty}}\right)\varepsilon.\]
Combined with \eqref{ineq-almost-there-na}, this completes the proof for generic and $h$-small sequences (we use that $\varphi$ is compactly supported in $V_v^\mathrm{an}$ in the case of good heights).

~

We now show how to deduce the full statement of Theorem~\ref{tm:equidistrib}, proceeding as in \cite[\S 5.5]{favregauthier}. Let us enumerate all irreducible hypersurfaces $(H_\ell)_\ell$ of $V$ that are defined over $\mathbb{K}$. We use the next lemma, see~e.g.~\cite[Lemma~5.12]{favregauthier}.

\begin{lemma}\label{lm:genericlike}
Take a sequence $(F_n)_n$ of Galois-invariant finite subsets of $V(\bar{\mathbb{Q}})$ with
\[\lim_{n\to\infty}\frac{\# (F_n\cap H_\ell)}{\# F_n}=0,\]
for any $\ell$.
Then, for any $\epsilon>0$, there exists a
sequence of sets $F'_{n,\epsilon} \subset F_k$ such that:
\begin{enumerate}
\item
$\# F'_{n,\epsilon} \geq (1-\epsilon) \# F_n$ for all $n$,
\item
$F'_{n,\epsilon}$ is  Galois-invariant,
\item
for any $\ell$ there exists $N(\ell)\geq1$, such that $F'_{n,\epsilon} \cap H_\ell = \varnothing$ for all $n\ge N(\ell)$.
\end{enumerate}
\end{lemma}
Fix $\epsilon>0$. The last condition of Lemma \ref{lm:genericlike} 
implies $F'_{n,\epsilon}$ to be generic. 
Now pick any continuous function $\varphi\in\mathscr{C}^0_c(V_v^\mathrm{an})$. The above implies that there is $n_0\geq1$ such that
\begin{align*}
\left|\int_{V_v^\mathrm{an}} \varphi\, \mu_{F_n,v} - \int_{V_v^\mathrm{an}}  \varphi \, \mu_v\right|
&\leq
\left|\int_{V_v^\mathrm{an}}  \varphi\,  \mu_{F_n,v} -  \int_{V_v^\mathrm{an}}  \varphi\,\mu_{F'_{n,\epsilon},v}\right| +\epsilon.
\end{align*}
We infer from Lemma~\ref{lm:genericlike} that
\begin{align*}
\left|\int_{V_v^\mathrm{an}} \varphi\, \mu_{F_n,v} - \int_{V_v^\mathrm{an}}  \varphi \, \mu_v\right|  \leq & 
\frac{1}{\# F_n}\int_{V_v^\mathrm{an}}|\varphi|\sum_{x\in F_n\setminus F'_{n,\epsilon}}\delta_x\\
&  + \left( \frac1{\# F'_{n,\epsilon}} - \frac1{\# F_n} \right)\int_{V_v^\mathrm{an}}|\varphi|\sum_{x\in F'_{n,\epsilon}}\delta_x\\
& \leq  2\,\epsilon\sup_{V_v^\mathrm{an}} |\varphi|~.
\end{align*}
This shows that for $n\geq n_0$, we have
\[\left|\int_{V_v^\mathrm{an}}\varphi\, \mu_{F_n,v}-\int_{V_v^\mathrm{an}}\varphi\, \mu_{v}\right|\leq \left(1+2\sup_{V_v^\mathrm{an}}|\varphi|\right)\epsilon,\]
which concludes the proof.

\subsection{Quasi-adelic measures on the projective line are good}\label{sec:qa}
Let $\mathbb{K}$ be a number field. We prove here that quasi-adelic measures on $\mathbb{P}^1$ as defined by Mavraki and Ye~\cite{mavraki-ye-quasiadelic}, and their induced height functions, satisfy the assumptions of Theorem~\ref{tm:equidistrib}. 

~

Let us defined quasi-adelic measures and their height functions following Mavraki and Ye. Pick $v\in M_\mathbb{K}$. Write $\log^+|\cdot|_v:=\log\max\{|\cdot|_v,1\}$ on $\mathbb{A}^1(\C_v)$. This function extends to $\mathbb{A}^{1,\mathrm{an}}_v$ and $dd^c \log^+|\cdot|_v=\delta_\infty-\lambda_v,$
as a function from $\mathbb{P}^{1,\mathrm{an}}_v$ to $\R_+\cup\{+\infty\}$, where $\lambda_v$ is the Lebesgue measure on $\{|z|_v=1\}$ if $v$ is archimedean, and $\lambda_v=\delta_{G,v}$ is the dirac mass at the Gau{\ss} point of $\mathbb{P}^{1,\mathrm{an}}_v$ otherwise.

A probability measure on $\mathbb{P}^{1,\mathrm{an}}_v$ has continuous potential if $\mu_v-\lambda_v=dd^c g_v$ for some $g_v\in\mathscr{C}^0(\mathbb{P}^{1,\mathrm{an}}_v,\mathbb{R})$. In this case, there is a unique function $g_{\mu_v}:\mathbb{P}^{1,\mathrm{an}}_v\to\R\cup\{+\infty\}$ such that $g_v=\log^+|\cdot|_v-g_{\mu_v}$ with the following normalization: if we let
\[G_{\mu_v}(x,y):=\left\{\begin{array}{ll}
g_{\mu_v}(x/y)+\log|y|_v & \text{for} \ (x,y)\in\C_v\times(\C_v\setminus\{0\}),\\
\log|x|_v-g_v(\infty) & \text{for} \ (x,y)\in(\C_v\setminus\{0\})\times\{0\},\\
-\infty & \text{for} \ (x,y)=(0,0),
\end{array}\right.\]
then the set $M_{\mu_v}:=\{(x,y)\in \C_v^2\, : \ G_{\mu_v}(x,y)\leq0\}$ has homogeneous logarithmic capacity $1$. Note that for $\alpha\in \C_v\setminus\{0\}$ and $(x,y)\in \C_v^2\setminus\{(0,0)\}$, we have 
\begin{center}
$G_{\mu_v}(\alpha x,\alpha y)=G_{\mu_v}(x,y)+\log|\alpha|_v$.
\end{center}
We then define the inner (resp. outer) radius of $\mu_v$ as
\[\left\{\begin{array}{l}
r_{\mathrm{in}}(\mu_v):=\sup\{r>0\, : \ \bar{D}_v(0,r)\times\bar{D}_v(0,r)\subset M_{\mu_v}\},\\
r_{\mathrm{out}}(\mu_v):=\inf\{r>0\, : \ M_{\mu_v}\subset \bar{D}_v(0,r)\times\bar{D}_v(0,r) \}
\end{array}\right.\]

\begin{definition}
We say that $\mu=\{\mu_v\}_{v\in M_\mathbb{K}}$ is a \emph{quasi-adelic} measure on $\mathbb{P}^1$ if 
\begin{itemize}
\item for any $v\in M_\mathbb{K}$, the measure $\mu_v$ has continuous potential,
\item both series $\sum_{v\in M_\mathbb{K}}|\log r_{\mathrm{in}}(\mu_v)|$ and $\sum_{v\in M_\mathbb{K}}|\log r_{\mathrm{out}}(\mu_v)|$ converge.
\end{itemize}
The \emph{height function} $h_\mu$ induced by a quasi-adelic measure $\mu=\{\mu_v\}_{v\in M_\mathbb{K}}$ is defined as
\[h_\mu(z):=\frac{1}{[\mathbb{K}:\mathbb{Q}]\#\mathsf{O}(x)}\sum_{v\in M_\mathbb{K}}\sum_{\sigma:\mathbb{K}(x)\hookrightarrow\C_v}G_{\mu_v}(\sigma(x),\sigma(y))),\]
where $(x,y)\in \mathbb{A}^2(\bar{\mathbb{Q}})\setminus\{(0,0)\}$ is any point with $z=[x:y]$.
\end{definition}
As the functions $G_{\mu_v}$ are homogeneous, the product formula implies the height function $h_\mu$ is well-defined and independent of the choice of $(x,y)$.

~

We prove here the following.
\begin{proposition}
Let $\mathbb{K}$ be a number field and let $\mu:=\{\mu_v\}_{v\in M_\mathbb{K}}$ be a quasi-adelic measure on $\mathbb{P}^1$ with induced height function $h_\mu$. Then $h_\mu$ is a good height function on $\mathbb{P}^1$ with induced global measure $\{\mu_v\}_{v\in M_\mathbb{K}}$.
\end{proposition}

\begin{proof}

Enumerate the places of $\mathbb{K}$ as $M_\mathbb{K}:=\{v_n, \ n\geq 0\}$ and define $X_n=\mathbb{P}^1$ and $\bar{L}_n$ as $\mathcal{O}_{\mathbb{P}^1}(1)$ endowed with the adelic continous semi-positive metrization $\{|\cdot|_{v,n}\}_{v\in M_\mathbb{K}}$ defined by the following conditions:
\begin{itemize}
\item for any $j\geq n+1$, the metric $|\cdot|_{v_j,n}$ is the usual naive metric on $\mathcal{O}_{\mathbb{P}^1}(1)$ at place $v_j$,
\item for any $j\leq n$, the metric is that induced by $g_{\mu_{v_j}}$,
\end{itemize}
so that for any $z\in\mathbb{P}^1(\bar{\mathbb{Q}})$ and any $(x,y)\in\mathbb{A}^2(\bar{\mathbb{Q}})\setminus\{(0,0)\}$ with $z=[x:y]$, we have
\[h_{\bar{L}_n}(z)=\frac{1}{[\mathbb{K}:\mathbb{Q}]\#\mathsf{O}(z)}\sum_{\sigma:\mathbb{K}(x)\hookrightarrow\C_v}\left(\sum_{j=0}^nG_{\mu_{v_j}}(\sigma(x),\sigma(y))+\sum_{\ell=n+1}^\infty\log\|\sigma(x),\sigma(y))\|_{v_\ell}\right).\]
For any $n\geq0$, we have $\mathrm{vol}(L_n)=1$ and for all $v\in M_\mathbb{K}$, we have $c_1(\bar{L}_n)_v=\mu_v$ if $n$ is large enough. All there is left to prove is that there is a $h$-small sequence and that the condition of pointwise approximation is satisfied on $\mathbb{P}^1$.

\medskip

The definitions of $r_{\mathrm{in}}(\mu_v)$ and $r_{\mathrm{out}}(\mu_v)$ and the product formula give
\[\sum_{j=n+1}^\infty\frac{\log (r_{\mathrm{in}}(\mu_v))}{[\mathbb{K}:\mathbb{Q}]}\leq h_{\bar{L}_n}-h_\mu\leq \sum_{j=n+1}^\infty\frac{\log (r_{\mathrm{out}}(\mu_v))}{[\mathbb{K}:\mathbb{Q}]},\]
on $\mathbb{P}^1(\bar{\mathbb{Q}})$. In particular, if we let
\[\varepsilon(n):=\sum_{j=n+1}^\infty\frac{1}{[\mathbb{K}:\mathbb{Q}]}\max\{|\log r_{\mathrm{in}}(\mu_v)|,|\log r_{\mathrm{out}}(\mu_v)|\},\]
then $\varepsilon(n)\to0$ as $n\to\infty$, by assumption, and 
\begin{align}
|h_{\bar{L}_n}-h_\mu|\leq \varepsilon(n), \quad \text{on} \ \mathbb{P}^1(\bar{\mathbb{Q}}).\label{eq:compare-qa}
\end{align}
We are thus left with justifying the existence of a small $h_\mu$-sequence to conclude the proof. Let $\bar{L}_n$ be a model of $\bar{L}_n$ over $\mathbb{P}^1_{\mathscr{O}_\mathbb{K}}$. By the arithmetic Hilbert-Samuel theorem as stated in \cite[Theorem~A]{Rumely-existence}, we have
\[0=-\sum_{j=0}^n\log \mathrm{cap}_{v_j}(M_{\mu_{v_j}})=\widehat{\mathrm{vol}}(\bar{L}_n)=\left(\bar{L}_n\right)^2,\]
since we assumed $\log \mathrm{cap}_{v_j}(M_{\mu_{v_j}})=0$ for all $j$.
In particular, $h_{\bar{L}_n}(\mathbb{P}^1)=0$ and the assumption $1$ of the definition of good height function is satsified.

Finally, for any $n\geq0$, there is a generic sequence $(F_{i,n})_i$ of Galois invariant finite subsets of $\mathbb{P}^1(\bar{\mathbb{Q}})$ such that $h_{\bar{L}_n}(F_{i,n})\to0$, as $i\to\infty$. In particular, By inequality \eqref{eq:compare-qa},
\[\limsup_{i\to\infty}|h_\mu(F_{i,n})|\leq\varepsilon(n),\]
and a diagonal extraction argument implies there exists $(\tilde{F}_i)_i$ which is generic and such that $\limsup_{i\to\infty}|h_\mu(\tilde{F}_i)|\to0$, and the proof is complete.
\end{proof}

\section{Equidistribution in families of dynamical systems}\label{sec:distrib-para}
We give here an application of Theorem~\ref{tm:equidistrib} in families of polarized endomorphisms with marked points. We begin with a general result and then we focus on special cases.

\subsection{Families of polarized endomorphisms}\label{sec:families}
Let $S$ be a smooth quasi-projective variety of dimension $p\geq1$ and let $\pi:\mathcal{X}\to S$ be a family of smooth projective varieties of dimension $k\geq p$. We say $(\mathcal{X},f,\mathcal{L})$ is a family of \emph{polarized endomorphisms} if $f:\mathcal{X}\to\mathcal{X}$ is a morphism with $\pi\circ f=\pi$ and if there is a relatively ample line bundle $\mathcal{L}$ on $\mathcal{X}$ and an integer $d\geq2$ such that $f^*\mathcal{L}\simeq \mathcal{L}^{\otimes d}$.
When $(\mathcal{X},f,\mathcal{L})$ is defined over $\bar{\mathbb{Q}}$, following Call and Silverman~\cite{CS-height}, for any $t\in S(\bar{\mathbb{Q}})$ we can define a \emph{canonical height} function for the restriction $f_t:X_t\to X_t$ of $f$ to the fiber $X_t:=\pi^{-1}\{t\}$ of $\pi:\mathcal{X}\to S$ as
\[\widehat{h}_{f_t}(x)=\lim_{n\to\infty}\frac{1}{d^n}h_{X_t,L_t}\left(f_t^{\circ n}(x)\right), \quad x\in X_t(\bar{\mathbb{Q}}).\]
If $\mathbb{L}$ is a finite extension of $\mathbb{K}$ such that $t\in S(\mathbb{L})$, the canonical height $\widehat{h}_{f_t}$ is induced by an adelic semi-positive continuous metrization $\{\|\cdot\|_{f_t,v}\}_{v\in M_\mathbb{L}}$ on the ample line bundle $\mathcal{L}_t$ of $X_t$. We then have
\begin{enumerate}
\item $\widehat{h}_{f_t}(x)\geq0$ for all $x\in X_t(\bar{\mathbb{Q}})$,
\item $\widehat{h}_{f_t}-h_{X_t,L_t}=O(1)$ on $ X_t(\bar{\mathbb{Q}})$, where $O(1)$ depends on $t$, and
\end{enumerate}
the function $\widehat{h}_{f_t}$ is characterized by those two properties. Moreover, by the Northcott property, for any $x\in X_t(\bar{\mathbb{Q}})$ we have
\begin{enumerate}
\item[3.] $\widehat{h}_{f_t}(x)=0$ if and only if $x$ is preperiodic under iteration of $f_t$.
\end{enumerate}

\begin{center}
$\dag$
\end{center}

When $\pi:\mathcal{X}\to S$, $\mathcal{L}$ and $f$ are all defined over $\C$, we can associate to $(\mathcal{X},f,\mathcal{L})$ a closed positive $(1,1)$-current on $\mathcal{X}(\C)$ with continuous potentials which we can define as
\begin{align}
\widehat{T}_{f}:=\lim_{n\to\infty}\frac{1}{d^n}\left(f^{\circ n}\right)^*(\iota^*\omega_{\mathbb{P}^N}),\label{def:T_f}
\end{align}
where $\iota:\mathcal{X}\hookrightarrow \mathbb{P}^N$ induces (a large power of) $\mathcal{L}\otimes\pi^*(\mathcal{M)}$, where $\mathcal{M}$ is ample on a projective model $\bar{S}$ of $S$. Moreover, the convergence towards $\widehat{T}_{f}$ is uniform local for potentials, see, e.g.~\cite[Proposition~2.6]{GV_Northcott}. This current $\widehat{T}_f$ restricts on fibers $X_t(\C)$ of $\pi$ as the \emph{Green current} of the endomorphism $f_t$ of $X_t(\C)$ which is polarized by $L_t$. Moreover, the closed positive $(k,k)$-current $\widehat{T}_{f}^{k}$ on $\mathcal{X}(\C)$ restricts to the fiber $X_t(\C)$ as $\mathrm{vol}(L_t)\cdot \mu_{f_t}$, where $\mu_{f_t}$ is the unique maximal entropy probability measure of $f_t$.

To a section $a:S\to\mathcal{X}$ of $\pi$ which is regular on the quasi-projective variety $S$, as introduced in \cite{GV_Northcott}, we associate a \emph{bifurcation current} defined as
\[T_{f,a}:=(\pi)_*\left(\widehat{T}_{f}\wedge[a(S)]\right).\]
Compare with~\cite{favredujardin} for the case when $\mathcal{X}$ has relative dimension $1$.

\subsection{Step 1 of the proof of Theorem~\ref{tm:principal}: reduction and definition of $(B_n,\bar{L}_n,\psi_n)$}

According to Theorem~\ref{tm:equidistrib}, all there is to do is to prove that $h_{f,\mathfrak{a}}$ is a good height and that, for any archimedean place $v\in M_\mathbb{K}$, the induced measure on $S_v^\mathrm{an}$ is $\mathrm{vol}_f(\mathfrak{a})^{-1}\mu_{f,\mathfrak{a},v}$.

\medskip
 
Recall that we are given $q\geq1$ sections $a_1,\ldots a_q:S\to\mathcal{X}$. We first justify that we can reduce to the case when $q=1$. Indeed, if $\pi_{[q]}:\mathcal{X}^{[q]}\to S$ is the $q$-fibered product of $\pi:\mathcal{X}\to S$, let $f^{[q]}:\mathcal{X}^{[q]}\to\mathcal{X}^{[q]}$ be defined by 
\[f^{[q]}(x)=(f_t(x_1),\ldots,f_t(x_q)), \quad x=(x_1,\ldots,x_q)\in\pi_{[q]}^{-1}\{t\}.\]
Then $(\mathcal{X}^{[q]},f^{[q]},\mathcal{L}^{[q]})$ is a family of polarized endomorphisms parametrized by $S$. Moreover, an easy computation gives, for all $x=(x_1,\ldots,x_q)\in\mathcal{X}^{[q]}$ with $\pi_{[q]}(x)=t$,
\begin{align}
\widehat{h}_{f^{[q]}_t}(x_1,\ldots,x_q)=\sum_{j=1}^q\widehat{h}_{f_t}(x_j)\geq0 \quad \text{and} \quad \widehat{h}_{f^{[q]}_t}\left( f_t^{[q]}(x)\right)=d\cdot \widehat{h}_{f^{[q]}_t}(x).\label{eq:fiberproduct}
\end{align}
Let $\mathfrak{a}:S\to\mathcal{X}^{[q]}$ be the section of $\pi_{[q]}$ induced by $a_1,\ldots,a_q$. As before, we define the bifurcation current of $\mathfrak{a}$ as
\[T_{f,\mathfrak{a}}:=(\pi_{[q]})_*\left(\widehat{T}_{f^{[q]}}\wedge[\mathfrak{a}(S)]\right).\]
An easy computation gives
\begin{align}
T_{f,\mathfrak{a}}=T_{f,a_1}+\cdots+T_{f,a_q}.\label{reduc}
\end{align}
Combining equations \eqref{reduc} and \eqref{eq:fiberproduct}, we deduce that $\mu_{f,\mathfrak{a},v}=T_{f,\mathfrak{a}}^p$ and that $h_{f,\mathfrak{a}}=\sum_{j=1}^qh_{f,a_j}$. In particular, up to replacing $f$ by $f^{[q]}$ and $a_1,\ldots,a_q$ by $\mathfrak{a}$, we can assume $q=1$.

~

Let $\iota:\mathcal{X}\hookrightarrow \mathbb{P}^{M_1}\times S\hookrightarrow \mathbb{P}^{M_1}\times \mathbb{P}^{M_2}\hookrightarrow \mathbb{P}^{N}$, where the last embedding is the Segre embedding.  Let also $\bar{S}$ be the closure of $S$ in $\mathbb{P}^{M_2}$ induced by this embedding. Up to taking a large multiple of $\mathcal{L}$, we can assume the first embedding $\iota_1$ is induced by $\mathcal{L}$ and
\[h_\iota=h_{\mathcal{X},\mathcal{L}}+h_S \ \text{ on} \ \mathcal{X}(\bar{\mathbb{Q}}),\]
where $h_S$ is the ample height on $\bar{S}$ induced by the second embedding.
For any $n\geq1$, define a section $\mathfrak{a}_n:S\to\mathcal{X}$ by
\[\mathfrak{a}_n(t):=f_t^{\circ n}(a(t)), \quad t\in S.\]
As $\mathfrak{a}_n$ is a section of $\pi:\mathcal{X}\to S$, it is injective. In particular, $\iota\circ \mathfrak{a}_n:S\to\mathbb{P}^N$ is finite.

We can assume $\mathfrak{a}$ extends as a morphism from $B_0:=\bar{S}$ to $\bar{\mathcal{X}}$, which is then generically finite. Set then $\psi_0:=\mathrm{id}$. We can define $B_n$ and $\psi_n$ inductively as follows: if $\iota\circ\mathfrak{a}_{n+1}$ extends as a morphism $B_n\to\mathbb{P}^N$, let $B_{n+1}:=B_n$ and $\psi_{n+1}=\psi_n$. Otherwise, let $p_{n+1}:B_{n+1}\to B_n$ be a finite sequence of blowups such that $\iota\circ \mathfrak{a}_{n+1}\circ \psi_n\circ p_{n+1}$ extends as a morphism $B_{n+1}\to\mathbb{P}^N$. Let then $\psi_{n+1}:=\psi_n\circ p_{n+1}$. The morphism $\iota\circ \mathfrak{a}_{n+1}\circ\psi_{n+1}:B_{n+1}\to\mathbb{P}^N$ is then generically finite by construction. We equip $\mathcal{O}_{\mathbb{P}^N}(1)$ with its standard metrization and denote it by $\bar{\mathcal{O}}(1)$. Define
\[\bar{L}_n:= \frac{1}{d^n}(\iota\circ \mathfrak{a}_n\circ\psi_n)^*\bar{\mathcal{O}}(1).\]
As $\bar{\mathcal{O}}(1)$ is ample and  $\iota\circ \mathfrak{a}_n\circ\psi_n$ is generically finite, $\bar{L}_n$ is big and nef. Moreover, by construction, the variety $B_n$ is defined over $\mathbb{K}$.

\subsection{Step 2 of the proof of Theorem~\ref{tm:principal}: convergence of the measures}\label{sec:CL-D}

Fix a place $v\in M_\mathbb{K}$. We now prove that the sequence $(c_1(\bar{L}_n)_v^p)_n$ of positive measures on $S_v^\mathrm{an}$ converges weakly to a positive measure $\mu_v$.

\medskip

Assume first $v$ is archimedean. We follow classical arguments we summarize here: we can write $d^{-1}f^*(\iota^*\omega_{\mathbb{P}^N})-\iota^*\omega_{\mathbb{P}^N}=dd^cu$,
where $u\in\mathscr{C}^\infty(\mathcal{X}_v^\mathrm{an})$. An easy induction gives, for any integer $n\geq1$,
\[\widehat{T}_{f}-\frac{1}{d^n}\left(f^{\circ n}\right)^*(\iota^*\omega_{\mathbb{P}^N})=\frac{1}{d^n}dd^cg\circ f^{\circ n} \quad \text{on} \ \ \mathcal{X}_v^\mathrm{an},\]
where $g=\sum_{j\geq 0} d^{-j}\cdot u\circ f^{\circ j}\in\mathscr{C}^0\left(\mathcal{X}_v^\mathrm{an}\right)$. Pulling back by $\mathfrak{a}:S\to\mathcal{X}$, we find $T_{f,\mathfrak{a}}-(\psi_n)_*c_1(\bar{L}_n)_v=d^{-n}\cdot dd^cg\circ f^{\circ n}$ on $S_v^\mathrm{an}$. In particular, $(\psi_n)_*c_1(\bar{L}_n)_v$ converges to $T_{f,\mathfrak{a}}$ with a local uniform convergence of potentials. Thus $((\psi_n)_*c_1(\bar{L}_n)_v^p)_n$ converges weakly on $S_v^\mathrm{an}$ to $\mu_{f,\mathfrak{a},v}$.

\medskip

\par When $v$ is non-archimedean, we rely on the work \cite{CL-D-currents} of Chambert-Loir and Ducros, and we employ freely their vocabulary. We follow the strategy used in the archimedean case, as presented in \cite{Boucksom-Eriksson}: let $\|\cdot\|_v$ be the naive (Fubini-Study) metric on $\mathcal{O}_{\mathbb{P}^N}(1)$ at the place $v$ so that the metric on $L_n$ is $d^{-n}(\iota\circ\mathfrak{a}_n\circ\psi_n)^*\|\cdot\|_v$. On the space of metrics on $\mathcal{L}$, the map $(d^{-1}f^*-\mathrm{id})$ is $\frac{C_K}{d}$-contracting over any compact subset $K$ of $\mathcal{X}^\mathrm{an}_v$. In particular, we have
\[\left|\frac{1}{d^n}\left(f^{\circ n}\right)^*\left(\iota^*\|\cdot\|_v\right)-\frac{1}{d^{n+1}}\left(f^{\circ (n+1)}\right)^*\left(\iota^*\|\cdot\|_v\right)\right|\leq\frac{C_K}{d^{n+1}},\]
over any compact subset $K\Subset\mathcal{X}^\mathrm{an}_v$. In particular, the sequence 
\[\left(\frac{1}{d^n}\left(f^{\circ n}\right)^*\left(\iota^*\|\cdot\|_v\right)\right)_n\]
converges uniformly over compact subsets of $\mathcal{X}^\mathrm{an}_v$ to a continuous and semi-positive metric $\|\cdot\|_{f,v}$ on $\mathcal{L}$ which satisfies
\[\left|\frac{1}{d^n}\left(f^{\circ n}\right)^*\left(\iota^*\|\cdot\|_v\right)-\|\cdot\|_{f,v}\right|\leq\frac{C_K'}{d^{n}},\]
over any compact subset $K\Subset\mathcal{X}^\mathrm{an}_v$. To conclude, we remark that on $S_v^\mathrm{an}$, one can write
\begin{align*}
(\psi_n)_*c_1(\bar{L}_n)_v & =\mathfrak{a}_0^*\left(\frac{1}{d^n}\left(f^{\circ n}\right)^*\iota^*c_1(\bar{\mathcal{O}}_{\mathbb{P}^N}(1))_v\right)\\
& =\mathfrak{a}_0^*\left(c_1(\mathcal{L},\|\cdot\|_{f,v})\right)+dd^cu_n,
\end{align*}
and the above implies that $(u_n)$ converges uniformly on any compact subset $K\Subset S_v^{\mathrm{an}}$ to the constant function $0\in\mathscr{C}^0(S_v^\mathrm{an})$ (which is obviously locally approachable). By \cite[Corollaire (5.6.5)]{CL-D-currents}, this implies
\[(\psi_n)_*c_1(\bar{L}_n)^p\to \mu_v:=\left(\mathfrak{a}_0^*\left(c_1(\mathcal{L},\|\cdot\|_{f,v})\right)\right)^p\]
in the weak sense of measures on $S_v^\mathrm{an}$, as granted.

\subsection{Step 3 of the proof of Theorem~\ref{tm:principal}: convergence of volumes}
By the above, $\bar{L}_n$ is an ample $\mathbb{Q}$-line bundle equipped with a semi-positive adelic continuous metrization. Moreover, we have $h_{\bar{L}_n}(t)\geq0$ for all $t\in B_n(\bar{\mathbb{Q}})$, so that Corollary~\ref{lm:Zhang} implies $h_{\bar{L}_n}(B_n)\geq0$. Next, we prove that $\mathrm{vol}(L_n)\to \mathrm{vol}_f(\mathfrak{a})$ as $n\to\infty$ to conclude. To do so, we rely on \cite{GV_Northcott}. 
\begin{lemma}\label{claim:volume}
There is $C_1>0$ depending only on $(\mathcal{X},f,\mathcal{L})$, $S$, $\mathfrak{a}$ and $\iota$ such that
\[\left|\mathrm{vol}(L_n)-\mathrm{vol}_{f}(\mathfrak{a})\right|\leq \frac{C_1}{d^n}.\]
Moreover, for any ample class $H$ on $X_0$, there is a constant $C_2>0$ depending only on $(\mathcal{X},f,\mathcal{L})$, $S$, $\mathfrak{a}$, $\iota$ and $H$ such that for any $n\geq1$,
\[\left(\psi_n^*H\cdot\left(L_n\right)^{p-1}\right)\leq C_2\]
\end{lemma}
\begin{proof}
Let us prove the claim, by computing masses of currents on $\mathcal{X}_v^{\mathrm{an}}$ with respect to the K\"ahler form $\iota^*(\omega_{\mathbb{P}^N})$. Then, one can interpret $\mathrm{vol}(L_n)$ as
\begin{align*}
\mathrm{vol}(L_n) =\int_{S_v^{\mathrm{an}}}\left(\frac{1}{d^{n}}(\iota\circ \mathfrak{a}_n)^*\omega_{\mathbb{P}^N}\right)^p& =\frac{1}{d^{np}}\int_{(\mathcal{X})_v^{\mathrm{an}}}[\mathfrak{a}(S_v^{\mathrm{an}})]\wedge\left( \left(f^{\circ n}\right)^*\iota^*(\omega_{\mathbb{P}^N})\right)^{p}.
\end{align*}
First, point 1 of Proposition~13 of \cite{GV_Northcott} implies $\|\widehat{T}_{f}\|$ is finite and
\[\left|d^n\|\widehat{T}_{f}\|-\|(f^{\circ n})^*(\iota^*\omega_{\mathbb{P}^N})\|\right|\leq C_0\]
for some constant $C_0$ depending only on $(\mathcal{X},f,\mathcal{L})$.
Then, the second point of Proposition 13 of \cite{GV_Northcott} rereads as
\[\left|\mathrm{vol}(L_n)-\mathrm{vol}_{f}(\mathfrak{a})\right|\leq C_0'\sum_{j=0}^{p-1}d^{n(j-p)}\left\|\widehat{T}_{f}^{\wedge j}\wedge[\mathfrak{a}(S_v^{\mathrm{an}})]\wedge\left( (f^{\circ n})^*\iota^*(\omega_{\mathbb{P}^N})\right)^{\wedge (p-j-1)}\right\|\]
for some constant $C_0'>0$ depending only on $(\mathcal{X},f,\mathcal{L})$. This gives
\[\left|\mathrm{vol}(L_n)-\mathrm{vol}_{f}(\mathfrak{a})\right|\leq C_0'\sum_{j=0}^{p-1}d^{n(j-p)}\|\widehat{T}_{f}\|^{j}\cdot \left(d^n\|\widehat{T}_{f}\|+C_0\right)^{p-j-1}\leq \frac{C_1}{d^{n}},\]
by Bezout Theorem.

\medskip

Let now $\omega_H$ be a smooth $(1,1)$-form on $S_v^\mathrm{an}$ cohomologous to $H$. Then
\[\left(H_n\cdot (L_n)^{p-1}\right)=\frac{1}{d^{n(p-1)}}\int_{\mathcal{X}_v^{\mathrm{an}}}[\mathfrak{a}(S_v^{\mathrm{an}})]\wedge\left( \left(f^{\circ n}\right)^*\iota^*(\omega_{\mathbb{P}^N})\right)^{p-1}\wedge(\pi^*\omega_H)\]
and as above, for any $n\geq1$, we have
\begin{align*}
\left(H_n\cdot (L_n)^{p-1}\right)\leq & \ C_0'\sum_{j=0}^{p-2}d^{n(j-p+1)}\left\|\widehat{T}_{f}^{j}\wedge[\mathfrak{a}(S_v^{\mathrm{an}})]\wedge\left( (f^{\circ n})^*\iota^*(\omega_{\mathbb{P}^N})\right)^{p-j-2}\wedge(\pi^*\omega_H)\right\|\\
& \quad + \left\|\widehat{T}_f^{p-1}\wedge[\mathfrak{a}(S_v^{\mathrm{an}})]\wedge(\pi^*\omega_H)\right\|,
\end{align*}
so that we deduce as in the proof of the first point of the Lemma that
\begin{align*}
\left(H_n\cdot (L_n)^{p-1}\right)\leq & \ C',
\end{align*}
for some constant $C'>0$ depending only on $(\mathcal{X},f,\mathcal{L})$, $\mathfrak{a}$, $\iota$ and $H$.
\end{proof}
\subsection{Step 4 of the proof of Theorem~\ref{tm:principal}: an upper bound on intersection numbers}

Let $\bar{M}_0$ be any ample adelic semi-positive line bundle on $B_0$. We prove here the next lemma.

\begin{lemma}
There is a constant $C\geq0$ such that for any $n\geq0$ and any $2\leq j\leq p+1$,
\[\left(\psi_n^*(\bar{M}_0)\right)^j\cdot\left(\bar{L}_n\right)^{p+1-j}\leq C.\]
\end{lemma}

\begin{proof}
 Up to changing the initial projective model $B_0$ of $S$, we can assume $B_0\setminus S$ is the support of an effective divisor $D$ of $B_0$ and that $\bar{E}:=f^*(\iota^*\bar{\mathcal{O}}(1))-d\iota^*\bar{\mathcal{O}}(1)=\pi^*(\bar{E}_0)$, for some adelic metrize line bundle $\bar{E}_0$ on $B_0$ which can be represented by a divisor supported on $\mathrm{supp}(D)$, where $\pi:\bar{\mathcal{X}}\to B_0$ is the extension of $\pi:\mathcal{X}\to S$. We may also assume $\mathfrak{a}_0$ extends as a morphism from $B_0$ to $\bar{\mathcal{X}}$. Let $\bar{\mathcal{X}}_{n,1}:=\bar{\mathcal{X}}\times_{B_0}B_n$ and there is a normal projective variety $\bar{\mathcal{X}}_n$, a birational morphism $\bar{\mathcal{X}}_n\to\bar{\mathcal{X}}_{n,1}$ and a projective morphism $\pi_n:\bar{\mathcal{X}}_n\to B_n$ which is flat over $\psi_n^{-1}(S)$ such that $f^{\circ j}$ extends as a morphism $F_{j,n}:\bar{\mathcal{X}}_n\to\bar{\mathcal{X}}$ for all $j\leq n$. Let also $\mathfrak{a}_{0,(n)}$ be the section of $\pi_n$ induced by $\mathfrak{a}_0$. Up to blowing up $B_n$, we can assume $\mathfrak{a}_{0,(n)}$ extends as a morphism from $B_n$ to $\bar{\mathcal{X}}_n$. Let $\Psi_n:\bar{\mathcal{X}}_n\to\bar{\mathcal{X}}$ be the birational morphism induced by this construction. Note that $\Psi_n$ is an isomorphism from $\pi_n^{-1}(\psi_n^{-1}(S))$ to $\pi^{-1}(S)$ and that $\pi\circ\Psi_n=\psi_n\circ \pi_n$ and $\Psi_n\circ \mathfrak{a}_{0,(n)}=\mathfrak{a}_0\circ\psi_n$.

Let $\bar{N}_0:=\mathfrak{a}_{0}^*\bar{E}$. By construction of $\bar{L}_n$ and $\bar{L}_{n+1}$ we can write
\begin{align*}
\bar{L}_{n+1}-p_{n+1}^*\bar{L}_n= & \, \frac{1}{d^{n+1}}\mathfrak{a}_{0,(n+1)}^*\left(F_{n,n+1}\right)^*\left(f^*\iota^*\bar{\mathcal{O}}(1)\right)\\
& \  -\frac{1}{d^n}\cdot \mathfrak{a}_{0,(n+1)}^*\left(F_{n,n+1}\right)^*\left(\iota^*\bar{\mathcal{O}}(1)\right)\\
=& \, \frac{1}{d^{n+1}}\mathfrak{a}_{0,(n+1)}^*\left(\Psi_n\right)^*\left(f^*\iota^*\bar{\mathcal{O}}(1)-d\cdot \iota^*\bar{\mathcal{O}}(1)\right)\\
& =\frac{1}{d^{n+1}}\mathfrak{a}_{0,(n+1)}^*\left(\Psi_n\right)^*(\bar{E})\\
= & \, \frac{1}{d^{n+1}}\psi_{n+1}^*\left(\bar{N}_0\right).
\end{align*}
If we let $\alpha(n):=\frac{1-d^{-n}}{d-1}$, using that $p_{n+1}\circ \psi_n=\psi_{n+1}$, an easy induction gives
\begin{align}
\bar{L}_{n}=\psi_{n}^*\left(\bar{L}_0+\alpha(n)\cdot \bar{N}_0\right).\label{eq:induction-formula}
\end{align}
In particular, if $\bar{\mathscr{N}}_0$ is a metrized line bundle on $\mathscr{B}_0$ which restricts to $\bar{N}_0$ on the special fiber of $\mathscr{B}_0\to\mathrm{Spec}(\mathscr{O}_\mathbb{K})$ andf if we let $I_n^j:=\left(\psi_n^*(\bar{M}_0)\right)^j\cdot\left(\bar{L}_n\right)^{p+1-j}$, using first the equation \eqref{eq:induction-formula} and then the projection formula, we find
\begin{align*}
I_{n}^j= &\left(\psi_{n}^*(\bar{M}_0)\right)^j\cdot\left(\psi_n^*\left(\bar{L}_0+\alpha(n)\cdot \bar{\mathscr{N}}_0\right)\right)^{p+1-j}\\
= & \left(\bar{M}_0\right)^j\cdot\left(\bar{L}_0+\alpha(n)\cdot \bar{\mathscr{N}}_0\right)^{p+1-j}\\
= & \sum_{\ell=0}^{p+1-j}{ p+1-j \choose \ell}  \left(\alpha(n)\right)^{p+1-j-\ell}\left(\bar{M}_0\right)^j\cdot\left(\bar{L}_0\right)^\ell\cdot\left(\bar{\mathscr{N}}_0\right)^{p+1-j-\ell}.
\end{align*}
Take $C_1\geq0$ such that $\left(\bar{M}_0\right)^j\cdot\left(\bar{L}_0\right)^\ell\cdot\left(\bar{\mathscr{N}}_0\right)^{p+1-j-\ell}\leq C_1$ for any $2\leq j\leq p+1$ and any $0\leq \ell\leq p+1-j$. As $\alpha(n)\leq 1$ for any $n\geq0$, the above gives
\begin{align*}
I_{n}^j\leq & \left(\bar{M}_0\right)^j\cdot\left(\bar{L}_0\right)^{p+1-j}+C_2,
\end{align*}
where $C_2$ depends only on $p$ and $C_1$. The conclusion of the lemma follows.
\end{proof}

\subsection{Step 5 of the proof of Theorem~\ref{tm:principal}: end of the proof}
All there is thus left to prove is that, for any given generic and $h$-small sequence $\{x_i\}_i$, the induced sequence $\varepsilon_n(\{x_i\}_i):=\limsup h_{\bar{L}_n}(x_i)$ converges to $0$ as $n\to\infty$. We rely on \cite{CS-height} and again on \cite{GV_Northcott} and we use Siu's classical bigness criterion as, e.g., in \cite[\S 7]{GOV2}.

The key point, in this step is the following in the spirit of \cite[Theorem~1.4]{Gao-Habegger}:
\begin{lemma}\label{lm:DGH}
If $\mathrm{vol}_f(\mathfrak{a})>0$, there is a non-empty Zariski open subset $U\subset \mathfrak{a}(S)$ and a constant $c>0$ depending only on $(\mathcal{X},f,\mathcal{L})$, $S$, $\mathfrak{a}$ and $\iota$ such that
\[h_S(\pi(x))\leq c\left(1+\widehat{h}_{f_{\pi(x)}}(x)\right), \quad x\in U(\bar{\mathbb{Q}}).\]
In particular, there is $c'>0$ depending only on $(\mathcal{X},f,\mathcal{L})$, $S$, $\mathfrak{a}$ and $\iota$ such that
\[\left|\widehat{h}_{f_{\pi(x)}}(x)-h_{\mathcal{X},\mathcal{L}}(x)\right|\leq c'\left(1+\widehat{h}_{f_{\pi(x)}}(x)\right), \quad x\in U(\bar{\mathbb{Q}}).\]
\end{lemma}

Befor proving this lemma, we can remark that the open set $U$ may be in fact $\mathfrak{a}(S)$, but that the strategy of the proof does not allow to prove it. It would be interesting to clarify the situation here, but it is not needed in the present proof.

\begin{proof}
Define a line bundle $H_n$ on $B_n$ by letting
\[H_n:=\psi_n^*\left(H\right),\]
where $H$ is the ample line bundle on $X_0$ inducing $h_S$. As $\psi_n$ is a birational morphism, $H_n$ is big and nef. By construction, $L_n$ is also big and nef. Fix $n_0\geq1$ large enough to that Lemma~\ref{claim:volume} gives
\[\mathrm{vol}(L_n)\geq \frac{1}{2}\mathrm{vol}_f(\mathfrak{a}), \ \text{ for all} \ n\geq n_0.\]
 Fix now an integer $M\geq1$ such that $M\mathrm{vol}_f(\mathfrak{a})>2pC_2$. By multilinearity of the intersection number, we find
\[\mathrm{vol}(ML_n)=M^p\mathrm{vol}(L_n)\geq\frac{M^p}{2}\mathrm{vol}_f(\mathfrak{a})>pM^{p-1}C_2\geq p\left(H_n\cdot (ML_n)^{p-1}\right),\]
so that by the classical bigness criterion of Siu~\cite[Theorem 2.2.15]{Laz}, $ML_n-H_n$ is big. In particular, there exists an integer $\ell\geq1$ such that $\ell(ML_n-H_n)$ is effective. According to \cite[Theorem B.3.2(e)]{Silvermandiophantine}, this implies there exists a Zariski open subset $U_n\subset B_n$ such that
\[h_{\ell (ML_n-H_n)}\geq O(1) \quad \text{on} \ U_n(\bar{\mathbb{Q}}).\]
Also, by functorial properties of Weil heights, see~e.g.~\cite[Theorem B.3.2(b-c)]{Silvermandiophantine}, 
\begin{align*}
h_{\ell (ML_n-H_n)} & =h_{\ell ML_n}-h_{\ell H_n}+O(1)=\ell\left(Mh_{\bar{L}_n}-h_{H_n}\right)+O(1)\\
& =\ell Mh_{\bar{L}_n}-\ell h_S\circ\psi_n+O(1),
\end{align*}
so that we have proved that
\begin{align*}
 Mh_{\bar{L}_n}\geq h_S\circ\psi_n+O(1), \quad \text{on} \ U_n(\bar{\mathbb{Q}}).
\end{align*}
Up to removing a Zariski closed subset of $U_n$, we can assume $U_n\subset \psi_n^{-1}(S)$. Since $\psi_n$ is an isomorphism from $U_n$ to $S_n:=\psi_n(U_n)\subset S$, this gives
\begin{align*}
h_S(t)\leq Mh_{\bar{L}_n}(\psi_n^{-1}(t))+N, \ t\in S_n(\bar{\mathbb{Q}}),
\end{align*}
for some constant $M\geq1$ independent of $n$ and some constant $N$ which depends a priori on $n$. In particular, when $t\in S_n(\bar{\mathbb{Q}})$ and $h_S(t)\to\infty$, we find
\begin{align}
\liminf_{h_S(t)\to\infty \atop t\in S_n(\bar{\mathbb{Q}})}\frac{h_{\bar{L}_n}(\psi_n^{-1}(t))}{h_S(t)}\geq\frac{1}{M}>0.\label{ineg-Siu-classical}
\end{align}
 According to~\cite[Theorem~3.1]{CS-height}, there is a constant $C_3\geq1$ depending only on $(\mathcal{X},f,\mathcal{L})$, $S$, $q$ and $\iota$ such that for all $t\in S(\bar{\mathbb{Q}})$ and all $x\in\mathcal{X}(\bar{\mathbb{Q}})$ with $\pi(x)=t$,
\begin{align}
\left|\widehat{h}_{f_t}(x)-h_{\iota^*\bar{\mathcal{O}}(1)}(x)\right|\leq C_3(h_S(t)+1).\label{CS-v0}
\end{align}
Evaluating the above for $x=\mathfrak{a}_n(t)$, we find
\[\left|\widehat{h}_{f_t}(\mathfrak{a}_n(t))-h_{\iota^*\bar{\mathcal{O}}(1)}(\mathfrak{a}_n(t))\right|\leq C_3(h_S(t)+1),\]
which, by invariance of $\widehat{h}_{f_t}$ and by definition of $\bar{L}_n$, gives
\begin{align}
\left|h_{f,\mathfrak{a}}(t)-h_{\bar{L}_n}(\psi_n^{-1}(t))\right|\leq \frac{C_3}{d^n}(h_S(t)+1), \quad t\in S(\bar{\mathbb{Q}}).\label{ineg:CS-almost-over}
\end{align}
We now fix $n_1\geq n_0$ such that $2MC_3\leq d^n$ for all $n\geq n_1$. Pick an integer $n\geq n_1$. Using \eqref{ineg-Siu-classical} and \eqref{ineg:CS-almost-over}, we deduce that
\begin{align*}
\liminf_{h_S(t)\to\infty \atop t\in S_n(\bar{\mathbb{Q}})}\frac{h_{f,\mathfrak{a}}(t)}{h_S(t)}\geq\liminf_{h_S(t)\to\infty \atop t\in S_n(\bar{\mathbb{Q}})}\frac{h_{\bar{L}_n}(\psi_n^{-1}(t))}{h_S(t)}-\frac{C_3}{d^n}\geq \frac{1}{M}-\frac{C_3}{d^n}\geq\frac{1}{2M}>0.
\end{align*}
This concludes the first assertion the proof of Lemma~\ref{lm:DGH}. The second follows using again \eqref{CS-v0}.
\end{proof}

To conclude the proof of Theorem~\ref{tm:equidistrib}, we now rewrite~\eqref{ineg:CS-almost-over} as
\[\left|h_{f,\mathfrak{a}}(t)-h_{\bar{L}_n}(\psi_n^{-1}(t))\right|\leq \frac{C_4}{d^n}(1+h_{f,\mathfrak{a}}(t)), \quad t\in S_n(\bar{\mathbb{Q}}),\]
where $C_4>0$ is a constant independent of $n$ and $t$. In particular, if $(F_i)_i$ is a generic and $h_{f,\mathfrak{a}}$-small sequence of finite Galois-invariant subsets of $S(\bar{\mathbb{Q}})$, then $F_i\subset S_n(\bar{\mathbb{Q}})$ for $i$ large enough and, using that $h_{\bar{L}_n}(B_n)\geq0$, we deduce that 
\[\limsup_{i\to\infty}h_{\bar{L}_n}(\psi_n^{-1}(F_i))-h_{\bar{L}_n}(B_n)\leq \limsup_{i\to\infty}h_{\bar{L}_n}(\psi_n^{-1}(F_i))\leq \frac{C_4}{d^n},\]
as required. We have proved that $h_{f,\mathfrak{a}}$ is a good height on $S$ with associated global measure $\{\mu_{f,a,v}\}_{v\in \mathbb{K}}$ and Theorem~\ref{tm:principal} follows from Theorem~\ref{tm:equidistrib}.

\section{Applications and sharpness of the assumptions}\label{sec:application}
We want finally to explore some specific cases. First, we consider the case of one parameter families of rational maps of $\mathbb{P}^1$. In a second time, we focus on the critical height on the moduli space of rational maps.
\subsection{One-dimensional families of rational maps}
In the case when $f:\mathbb{P}^1\times S\to\mathbb{P}^1\times S$ is a family of rational maps of $\mathbb{P}^1$ parametrized by a smooth quasi-projective curve, Theorem~\ref{tm:principal} reduces to the following:

\begin{theorem}\label{tm:curves}
Let $f:\mathbb{P}^1\times S\to\mathbb{P}^1\times S$ be a family of degree $d\geq1$ rational maps parametrized by a smooth quasi-projective curve $S$ and let $a:S\to\mathbb{P}^1$ be a rational function, all defined over a number field $\mathbb{K}$. Assume $\widehat{h}_{f_\eta}(a_\eta)>0$, where $\eta$ is the generic point of $S$. Then the set $\mathrm{Preper}(f,a):=\{t\in S(\bar{\mathbb{Q}})\, : \ \widehat{h}_{f_t}(a(t))=0\}$ is infinite.

Moreover, for any $v\in M_\mathbb{K}$ and for any non-repeating sequence $F_n\subset S(\bar{\mathbb{Q}})$ of Galois-invariant finite sets such that
\begin{align*}
\frac{1}{\# F_n}\sum_{t\in F_n}\widehat{h}_{f_t}(a(t))\to0, \quad \text{as} \ n\to\infty,
\end{align*}
the sequence $\mu_{F_n,v}$ of probability measures on $S_v^\mathrm{an}$ equidistributed on $F_n$ converges weakly to $\frac{1}{\widehat{h}_{f_\eta}(a_\eta)}\mu_{f,a,v}$ as $n\to\infty$.
\end{theorem}

\begin{proof}
Choose any $v\in M_\mathbb{K}$ archimedean. In this case, the author and Vigny~\cite[Theorem B]{GV_Northcott} proved that
\[\widehat{h}_{f_\eta}(a_\eta)=\int_{S_v^\mathrm{an}}\mu_{f,a,v}.\]
In particular, when $S$, $f$ and $a$ are defined over a number field $\mathbb{K}$, the assumption that $\mu_{f,a,v}>0$ for some archimedean $v\in M_\mathbb{K}$ is satisfied whenever $\widehat{h}_{f_\eta}(a_\eta)>0$. 
To be able to apply Theorem~\ref{tm:principal}, it is thus sufficient to be able to produce a sequence $t_n\in S(\bar{\mathbb{Q}})$ such that $\widehat{h}_{f_{t_n}}(a(t_n))=0$ for all $n$ and such that, if $\mathsf{O}(t_n)$ is the Galois orbit of $t_n$, then $\mathsf{O}(t_n)\cap\mathsf{O}(t_m)=\varnothing$ for all $n\neq m$. This is now classical and it can be done using Montel's Theorem.

We reproduce here the argument for completeness: let $v\in M_\mathbb{K}$ be archimedean. Let $U\subset S_v^\mathrm{an}$ be any euclidean open set such that $\mu_{f,a,v}(U)>0$. By e.g.~\cite[Theorem~9.1]{DeMarco2} or \cite[Proposition-Definition~3.1]{favredujardin}, this is equivalent to the fact that the sequence of rational functions $(a_n)_n$ defined by $a_n(t):=f_t^{\circ n}(a(t))$ is not a normal family on $U$. Up to reducing $U$, by the Implicit Function Theorem, we can assume there exists $N\geq3$ and holomorphic function $z:U\to\mathbb{P}^{1,\mathrm{an}}_v$ such that $N$ is minimal such that$f_t^{\circ n}z(t)=z(t)$ for all $t\in U$, and $z(t)$ is repelling for $f_t$, i.e. $|(f_t)^{\circ n}(z(t))|>1$. By Montel's Theorem, one can define inductively $t_{j+1}\in U\setminus\{t_\ell, \ell\leq j\}$ such that for all $j\geq1$, there is $n(j)\geq 1$ and $a_{n(j)}(t_j)\in \{z(t_j),f_{t_j}(z(t_j)),f_{t_j}^{\circ 2}(z(t_j))\}$. We thus have defined an infinite sequence $(t_n)$ of parameters for which $a(t_n)$ is preperiodic. In particular, we have $\widehat{h}_{f_{t_n}}(a(t_n))=0$ for all $n\geq1$ and the proof is complete.
\end{proof}

\subsection{Sharpness of the assumptions}
We now explain why the assumptions are sharp. When $(\mathcal{X},f,\mathcal{L})$ is a family of polarized endomorphisms of degree $d$ parametrized by a smooth complex quasi-projective curve $S$, we still have
\[\widehat{h}_{f_\eta}(a_\eta)=\int_{S(\C)}\mu_{f,a}=\int_{\mathcal{X}(\C)}\widehat{T}_f\wedge[a(S(\C))]\]
for any regular section $a:S\to\mathcal{X}$ by \cite[Theorem~B]{GV_Northcott}. However when the relative dimension of $\pi:\mathcal{X}\to S$ is $k>1$, the most probable situation is that there are at most finitely many $t\in S(\bar{\mathbb{Q}})$ such that $\widehat{h}_{f_t}(a(t))=0$. 

\begin{lemma}\label{lm:case-empty}
There is a family of polarized endomorphisms $(\mathcal{X},f,\mathcal{L})$ parametrized by $\mathbb{A}^1$ of relative dimension $2$ and a section $a:\mathbb{A}^1\to\mathcal{X}$, all defined over $\mathbb{Q}$, such that 
\[\widehat{h}_{f_\eta}(a_\eta)=1 \quad \text{and} \quad \mathrm{Preper}(f,a):=\{t\in \mathbb{A}^1(\bar{\mathbb{Q}})\, : \ \widehat{h}_{f_t}(a(t))=0\} =\varnothing.\]
\end{lemma}
\begin{proof}
let $f:(\mathbb{P}^1)^2\times\mathbb{A}^1\to(\mathbb{P}^1)^2\times\mathbb{A}^1$ be defined by
\[f_t(z,w)=(z^2+t,w^2+t), \quad (z,w,t)\in\mathbb{A}^1\times\mathbb{A}^1\times\mathbb{A}^1.\]
Define now a section $a:\mathbb{A}^1\to\mathbb{A}^1\times\mathbb{A}^1\times\mathbb{A}^1$ of the canonical projection $\pi:\mathbb{A}^1\times\mathbb{A}^1\times\mathbb{A}^1\to\mathbb{A}^1$ by letting $a(t):=(t,0,4)$ for all $t\in\mathbb{A}^1$. 
Write $|\cdot|:=|\cdot|_\infty$. For any $t\in\mathbb{A}^1$, let $p_t(z):=z^2+t$ we define the 
\[G_{t}(z):=\lim_{n\to\infty}\frac{1}{2^n}\log^+|p_t^{\circ n}(z)|, \quad (t,z)\in\C\times\C.\]
Then $\mu_{f,a}=dd^c_t\left(G_{t}(0)+G_{t}(4)\right)$. Since for $u\in\{0;4\}$ we have
\[G_{t}(u)=\frac{1}{2}\log^+|t|+O(1), \quad \text{as} \ |t|\to\infty,\]
the measure $\mu_{f,a}$ is a probability measure on $\mathbb{A}^{1,\mathrm{an}}$, whence $\widehat{h}_{f_\eta}(a_\eta)=1$.
Also, by an elementary computation, we have
\[h_{f,a}(t)=\widehat{h}_{p_t}(0)+\widehat{h}_{p_t}(4), \quad t\in \mathbb{A}^1(\bar{\mathbb{Q}}).\]
In particular, if $h_{f,a}(t)=0$, then $G_{t}(0)=G_{t}(4)=0$. The condition $G_{t}(0)=0$ implies $|t|\leq 2$ and by~\cite[Lemma~7]{Buff-PCF-unicritical} the condition $G_{t}(4)=0$ implies that
\begin{itemize}
\item either $|t|\leq 2$ and $|t\cdot p_t^{\circ n}(4)|\leq 2$ for all $n\geq0$,
\item or $|t|>2$ and $|t\cdot p_t^{\circ n}(4)|<1$ for all $n\geq0$.
\end{itemize}
The second condition is empty since, for $n=0$, this implies $2<1/4$.
For $n=0$, the first condition implies $|t|\leq 1/2$ and for $n=1$, it gives $|t|\leq 5/32$. In particular, the polynomial $p_t$ has an attracting fixed point and the only case when $\widehat{h}_{f_t}(0)=0$ is the case $t=0$. Finally, for $t=0$, $G_0(4)=\log|4|>0$ ending the proof.
\end{proof}

\medskip

We thus can ask whether the condition of existence of a Zariski dense set of small points is reasonable in families with relative dimension $>1$. Following the proof of Theorem~0.1 of \cite{Dujardin2012} exactly as adapted in \cite{Gauthier-smooth-bif}, we can prove the next proposition.

\begin{proposition}\label{prop:para-distrib}
Let $(\mathcal{X},f,\mathcal{L})$ be a family of polarized endomorphisms of degree $d$ parametrized by a smooth complex quasi-projective variety $S$ and let $q\geq1$ be an integer. Assume $\dim S=qk$, where $k$ is the relative dimension of $\mathcal{X}$. Assume there are $q$ sections $a_1,\ldots,a_q:S\to\mathcal{X}$ with $\mu_{f,\mathfrak{a}}>0$.
Then set
\[\mathrm{Preper}(f,a_1,\ldots,a_q):=\{t\in S(\bar{\mathbb{Q}})\, : \ \widehat{h}_{f_t}(a_1(t))=\cdots=\widehat{h}_{f_t}(a_q(t))=0\}\]
is Zariski dense in $S(\bar{\mathbb{Q}})$. In particular, if $S$, $(\mathcal{X},f,\mathcal{L})$ and $a_1,\ldots,a_q$ are defined over a number field, assumption $2$ of Theorem~\ref{tm:principal} is satisfied.
\end{proposition}

We omit the proof since it copies verbatim that of \cite[Theorem~2.2]{Gauthier-smooth-bif}.

~

Finally, fix $k\geq1$ and $d\geq2$ and let $(\mathcal{X},f,\mathcal{L})$ be a family of polarized endomorphisms of degree $d$, parametrized by a smooth complex quasi-projective variety $S$ with $\dim S>1$, where $\mathcal{X}$ has relative dimension $k$ and let $a_1,\ldots,a_q:S\to\mathcal{X}$ be $q\geq1$ sections. Given a K\"ahler form $\omega$ on $S$ which is cohomologous to $c_1(M)$ with $M$ ample on $S$, \cite[Theorem~B]{GV_Northcott} reads as
\[\int_{S(\C)}\left(T_{f,a_1}+\cdots+T_{f,a_q}\right)\wedge \omega_S^{\dim S-1}=\sum_{j=1}^q\widehat{h}_{f_\eta}(a_{j,\eta}).\]
The hypothesis that $\mu_{f,a,v}>0$ for some archimedean $v\in M_\mathbb{K}$ is thus stronger than only assuming $\sum_{j=1}^q\widehat{h}_{f_\eta}(a_{j,\eta})>0$, which -- by the above formula -- is equivalent to assuming that $T_{f,a_1}+\cdots +T_{f,a_q}>0$.
We can prove the following.

\begin{lemma}
There is a family of $(\mathcal{X},f,\mathcal{L})$ of polarized endomorphisms parametrized by $\mathbb{A}^2$ and a section $a:\mathbb{A}^2\to\mathcal{X}$, all defined over $\mathbb{Q}$ where $\mathcal{X}$ has relative dimension $1$ and such that:
\begin{enumerate}
\item if $\omega_{\mathbb{P}^2}$ is the Fubini-Study form of $\mathbb{P}^2(\C)$, the current $T_{f,a}$ satisfies
\[\int_{\C^2}T_{f,a}\wedge\omega_{\mathbb{P}^2}=1.\]
\item the bifurcation measure $\mu_{f,a}$ vanishes identically,
\item the set $\mathrm{Preper}(f,a)$ is Zariski dense in $\mathbb{A}^2(\bar{\mathbb{Q}})$.
\end{enumerate}
\end{lemma}

The idea behind the proof is that, in relative dimension $k$, the current $\widehat{T}_f^{k+1}$ vanishes identically. In particular, the current $T_{f,a}^{k+1}$ also vanishes.

\begin{proof}
Define a family of degree $4$ polynomials $f:\mathbb{P}^1\times\mathbb{A}^2\to\mathbb{P}^1\times\mathbb{A}^2$ by letting
\[f_{s,t}(z)=\frac{1}{4}z^4-\frac{2}{3}sz^3+\frac{s^2}{2}z^2+s^4, \quad z\in\mathbb{A}^1 \ \text{and} \ (s,t)\in\mathbb{A}^2,\]
and define $a:\mathbb{A}^2\to\mathbb{A}^1\times\mathbb{A}^2$ by $a(s,t)=\left(s,(s,t)\right)$. Let $|\cdot|:=|\cdot|_\infty$.
As in the proof of Lemma~\ref{lm:case-empty}, define $G:\C\times\C^2\to\R_+$ by letting
\[G_{s,t}(z):=\lim_{n\to\infty}\frac{1}{4^n}\log^+|f_{s,t}^{\circ n}(z)|, \ (s,t)\in\C^2, \ z\in\C.\]
We then have $T_{f,a}=dd^cG_{s,t}(s)$ and $\mu_{f,a}=\left(dd^cG_{s,t}(s)\right)^2$.

We now follows arguments from \cite[\S5--6]{favredujardin}.
Let us first justify that $\mu_{f,a}=0$. This follows from Bedford-Taylor theory. Let $g(s,t):=G_{s,t}(s)$ so that the current $T_{f,a}$ is $dd^cg$. As $g\geq0$, we have $g=\lim_{n\to\infty}\max\{g,\frac{1}{n}\}$ and since $\mathrm{supp}(dd^cg)\subset \{g=0\}$ and $\mathrm{supp}(dd^c\max\{g,1/n\})\subset \{g=1/n\}$, we have
\[\mu_{f,a}=\lim_{n\to\infty}dd^cg\wedge dd^c\max\left\{g,\frac{1}{n}\right\}=0.\]
Since $g(s,t)\leq \log^+\max\{|s|,|t|\}+O(1)$ as $\|(s,t)\|\to\infty$, we have 
\[\int_{\C^2}T_{f,a}\wedge\omega_{\mathbb{P}^2}\leq 1.\]
In particular, by Siu's extension Theorem, the trivial extension of $T_{f,a}$ to $\mathbb{P}^2(\C)$ is a closed positive $(1,1)$-current $S$ which decomposes at $\tilde{T}+\alpha[L_\infty]$, where $[L_\infty]$ is the integration current on the lien at infinity and $\tilde{T}\wedge \omega_{\mathbb{P}^2}$ gives no mass to $L_\infty$. But \cite[Theorem~4.2]{BB2}  implies that $\alpha=0$, whence
\[\int_{\C^2}T_{f,a}\wedge\omega_{\mathbb{P}^2}= 1.\]
To prove the last assertion, for $n>m\geq0$, we let
\[\mathrm{Preper}(n,m):=\{(s,t)\in\C^2\, : \ f_{s,t}^{\circ n}(s)=f_{s,t}^{\circ m}(s)\}.\]
For $n>m\geq0$, the set $\mathrm{Preper}(n,m)$ is a (possibly reducible) plane curve of degree $4^n$ which is defined over $\mathbb{Q}$. In particular, the set $\mathrm{Preper}(n,m)(\bar{\mathbb{Q}})$ is infinite. Also \cite[Theorem~1]{favredujardin} implies that for any sequence $\{m(n)\}_n$ with $0\leq m(n)<n$,
\[T_{f,a}=\lim_{n\to\infty}\frac{1}{4^n}[\mathrm{Preper}(n,m(n))]\]
in the weak sense of currents. As $T_{f,a}=dd^cg$ where $g$ is continuous, the set $\mathrm{Preper}(f,a)$ is Zariski dense.
\end{proof}

\subsection{In the moduli space of degree $d$ rational maps}\label{sec:PCF}

We finally focus on the case of the moduli space $\mathcal{M}_d$ of degree $d$ rational maps and we give a very quick proof of Theorem~\ref{cor:ratd}. The variety $\mathcal{M}_d$ is the space of $\mathrm{PGL}(2)$ conjugacy classes of rational maps of degree $d$. By~\cite{Silverman-Space-rat}, the variety $\mathcal{M}_d$ is irreducible, affine has dimension $2d-2$, and is defined over $\mathbb{Q}$.

\medskip

The good setting to apply Theorem~\ref{tm:principal} is actually the \emph{critically marked moduli space} $\mathcal{M}_d^{\mathrm{cm}}$. As in \cite{buffepstein}, define first
\[\mathrm{Rat}_d^{\mathrm{cm}}:=\{(f,c_1,\ldots,c_{2d-2})\in\mathrm{Rat}_d\times(\mathbb{P}^1)^{2d-2}\, :\ \mathrm{Crit}(f)=\sum_j[c_j]\},\]
where $\mathrm{Crit}(f)$ stands for the \emph{critical divisor} of $f$. The space $\mathrm{Rat}_d^\mathrm{cm}$ is an quasi-projective variety of dimension $2d+1$ which is a finite branched cover of $\mathrm{Rat}_d$. We then define the  \emph{critically marked moduli space} $\mathcal{M}_d^{\mathrm{cm}}$ as
\[\mathcal{M}_d^\mathrm{cm}:=\mathrm{Rat}_d^\mathrm{cm}/\mathrm{PGL}(2),\]
where $\mathrm{PGL}(2)$ acts by $\phi\cdot (f,c)=(\phi\circ f\circ \phi^{-1},\phi(c_1),\ldots,\phi(c_{2d-2}))$, and the quotient is geometric in the sense of Invariant Geometric Theory as in \cite{Silverman-Space-rat}. Again, it is an irreducible affine variety defined over $\mathbb{Q}$. Moreover, we can directly apply Theorem~\ref{tm:principal} to the good height function $h_{\mathrm{Crit}}:\mathcal{M}_d^\mathrm{cm}\to\R_+$ defined by
\[h_{\mathrm{Crit}}(f,c_1,\ldots,c_{2d-2})=\sum_{j=1}^{2d-2}\widehat{h}_f(c_j),\]
for all $(f,c_1,\ldots,c_{2d-2})\in \mathcal{M}_d^\mathrm{cm}(\bar{\mathbb{Q}})$. Indeed, we have a natural map $f:(\mathbb{P}^1)^{2d-2}\times\mathcal{M}_d^\mathrm{cm}\to(\mathbb{P}^1)^{2d-2}\times\mathcal{M}_d^\mathrm{cm}$ together with a section $\mathfrak{c}$ defined as 
\[\mathfrak{c}:\{(f,c_1\ldots,c_{2d-2})\}\mapsto \left((c_1,\ldots,c_{2d-2}),\{(f,c_1\ldots,c_{2d-2})\}\right).\] The current $T_{f^{[2d-2]},\mathfrak{c}}$ is then the bifurcation current $T_\bif$ of the family.

~

We now justify quickly why we are in the domain of application of Theorem~\ref{tm:principal}. For any irreducible subvariety $V\subset\mathcal{M}_d$, the measure $\mu_{\mathrm{bif},V}:=T_{\mathrm{bif}}^{\dim V}$ is non zero if and only if $V$ does not coincide with the curve $\mathcal{L}_d$ of flexible Latt\`es maps, by \cite[Lemma~6.8]{GOV2}. Here $\mathcal{L}_d$ is, when $d=N^2$, the family of maps induced by the multiplication by $N$ on a non-isotrivial elliptic surface $\mathcal{E}\to S$. 
In particular, $T_\bif^{2d-2}>0$ on $\mathcal{M}_d^{\mathrm{cm}}$.

\medskip

As $\dim\mathcal{M}_d^\mathrm{cm}=2d-2$, Proposition~\ref{prop:para-distrib} implies the set of $h_{f,\mathrm{crit}}$-small points form a Zariski dense subset of $\mathcal{M}_d^\mathrm{cm}(\bar{\mathbb{Q}})$. In particular, we are in position to apply Theorem~\ref{tm:principal} in the family $\mathcal{M}_d^{\mathrm{cm}}$. To conclude the proof of Theorem~\ref{cor:ratd}, we just need to recall that
\begin{enumerate}
\item the canonical projection $p:\mathcal{M}_d^\mathrm{cm}\to\mathcal{M}_d$
is a finite branched cover,
\item a conjugacy class $\{f\}$ is post-critically finite (PCF) iff and only if $\{f,c_1,\ldots,c_{2d-2}\}$ is $h_{f,\mathrm{crit}}$-small,
\item the bifurcation measures $\mu_\bif$ and $\mu_{\bif,\mathrm{cm}}$ respectively of $\mathcal{M}_d(\C)$ and of $\mathcal{M}_d^{\mathrm{cm}}(\C)$ are related by $\mu_{\bif,\mathrm{cm}}=p^*(\mu_\bif)$.
\end{enumerate}
In particular, we deduce that for any sequence $F_n\subset\mathcal{M}_d(\bar{\mathbb{Q}})$ of Galois-invariant finite sets of post-critically finite parameters such that
\[\frac{\# (F_n\cap H)}{\# F_n}\to0, \quad \text{as} \ n\to\infty,\]
for any hypersurface $H$ of $\mathcal{M}_d$ which is defined over $\mathbb{Q}$, for any place $v\in M_\mathbb{Q}$ the sequence of probability measures $\frac{1}{\# F_n}\sum_{\{f\}\in F_n}\delta_{\{f\}}$ on $\mathcal{M}_{d,v}^\mathrm{an}$ converges weakly to $\mathrm{vol}(\mu_\bif)^{-1}\mu_{\bif,v}$ in the weak sense of probability measures on $\mathcal{M}_{d,v}^\mathrm{an}$.

\bibliographystyle{short}
\bibliography{biblio}

\end{document}